\documentclass[a4paper, 12pt, oneside]{article}
\usepackage{graphicx}
\usepackage{amsmath}
\usepackage{amssymb}
\usepackage{amsfonts}
\usepackage{amsthm}
\usepackage{longtable}
\usepackage{multirow}
\usepackage{color}
\usepackage{makeidx}

\makeindex
\title{Regular Polyhedra of Index Two, II}
\author{Anthony M. Cutler\\
Northeastern University\\
Boston, MA 02115, USA\\
{\small amcutler@msn.com}}
\date{ }

\newtheorem{lemma}{Lemma}[section] 
\newtheorem{theorem}[lemma]{Theorem}

\theoremstyle{definition}
\newtheorem*{acknowledgement}{Acknowledgement}

\renewcommand{\Gamma}{\varGamma} 
\renewcommand{\epsilon}{\varepsilon}

\renewcommand{\geq}{\geqslant} 

\newcommand{\E}{\mathbb{E}^3}

\begin{document}

\maketitle

\begin{abstract}
\noindent
A polyhedron in Euclidean $3$-space $\E$ is called a \textit{regular polyhedron of index $2$} if it is combinatorially regular and its geometric symmetry group has index $2$ in its combinatorial automorphism group; thus its automorphism group is flag-transitive but its symmetry group has two flag orbits. The present paper completes the classification of finite regular polyhedra of index $2$ in $\E$. In particular, this paper enumerates the regular polyhedra of index $2$ with vertices on one orbit under the symmetry group. There are ten such polyhedra.
\end{abstract}

\section{Introduction}
\label{intro}

The definition of a finite regular polyhedron in Euclidean $3$-space has changed over time, and as a result so has the number of such polyhedra, as Coxeter~\cite{cox} and Gr\"unbaum~\cite{grsame}, among other authors, have pointed out. Kepler realized that some star polyhedra, built by stellation from the dodecahedron or icosahedron, are regular if the convexity requirement is dropped. Poinsot later rediscovered Kepler's two polyhedra, and used star vertex-figures to construct the remaining two regular star polyhedra. Thus as the notion of a face (or vertex-figure) was expanded to include intersecting boundary edges, the four Kepler-Poinsot polyhedra joined the five classical Platonic solids as regular polyhedra.

In the 1920's and 1930's, Coxeter and Petrie~\cite{cox,crsp} discovered that there are additional regular polyhedra if planarity is removed as a requirement of vertex-figures; these polyhedra, known as the Petrie-Coxeter polyhedra, are infinite and have convex faces but skew vertex-figures. Petrie also was the first to observe that the family of `Petrie polygons' of a regular map (cell-decomposition) on a closed surface (real $2$-manifold) usually forms another regular map on another closed surface, called the `Petrie dual' of the original map; here, a Petrie polygon is a path along edges such that every two, but no three, consecutive edges lie in the same face. In the 1970's, Gr\"unbaum~\cite{gr1} restored the symmetry in the definition by admitting the possibility that a polyhedron has skew faces. Thus a polyhedron became a geometric graph in space equipped with a distinguished (connected) family of simple edge-cycles, called faces, independent of whether or not topological discs could be spanned into these cycles to result in a surface free of self-intersections. As the generally accepted definition of regularity changed, the Petrie duals of the known regular polyhedra were recognized as regular polyhedra, bringing the number of finite regular polyhedra to eighteen (see \cite{ms1} and \cite[\S 7E]{arp}). This graph-theoretic approach to polyhedral structures and symmetry has driven many recent developments. Gr\"unbaum~\cite{gr1} began, and Dress~\cite{d1,d2} completed, the classification of the forty-eight infinite regular polyhedra with non-planar polygons (apeirogons) as faces. Schulte~\cite{chi1,chi2} has shown that discrete chiral polyhedra in Euclidean $3$-space belong to six families, three with finite faces and three with infinite helical faces. McMullen~\cite{pm1,pm,pm2,ms3} has obtained classification results for polytopes in higher dimensions. Classification results have also been obtained for polyhedra with specific symmetry properties; see \cite{far,grhol,mw5,ps}, and Martini \cite{mar1,mar2} for surveys of convex polyhedra.

These changes in perception of a polyhedron emphasized the combinatorial aspects of the structure over geometric ones, and went hand-in-hand with the gradual change in defining the term `regular', eventually leading to the modern definition in terms of a flag-transitive symmetry group. The combinatorial structure of a finite polyhedron can be represented by a finite map, which is a cell-decomposition of a closed compact surface into topological polygons. Coxeter~\cite{cox} and Coxeter-Moser~\cite{cm} listed all regular maps of genus less than $3$. This was extended by Sherk~\cite{she} and Garbe~\cite{gar}, who classified the orientable maps of genus $3$, and $4, 5,$ and $6$, respectively. Subsequently, Conder and Dobesanyi~\cite{codo} using computer search expanded this list to regular maps of genus up to $15$, and more recently Conder~\cite{con} has listed all regular maps of genus up to $100$ in the orientable case, and $200$ in the nonorientable case.

Less is known about the geometric realizations of these combinatoric structures. Not every regular map has a polyhedral realization. Additionally, two or more geometrically dissimilar polyhedra may be realized from the same regular map. Every geometric polyhedron realizing a regular map has associated with it the (geometric) symmetry group, which can be viewed as a subgroup of the (combinatorial) automorphism group of the underlying regular map. When the index of this subgroup is finite, it is called the index of the polyhedron. The $18$ polyhedra mentioned above are the only finite regular polyhedra of index $1$. Wills~\cite{wills} showed that there are precisely five orientable finite regular polyhedra of index $2$ with planar faces. For figures of these, see Richter~\cite{rich}, and also \cite[Figs. 4.4a, 6.4c]{cox}, \cite[Figs. 45, 53]{clm}, \cite[$De_{2}f_{2}$ (Plate XI), $Ef_{1}g_{1}$ (Plate IX)]{cdfp}, \cite[Fig. 26, Tafel VIII]{bru} and \cite{sw}. These polyhedra are described in \cite{bru}, \cite[\S 6.4]{cox} and \cite{gor}, or can be obtained from those by polarity (see \cite{gsh}). E. Schulte and the author \cite{cs} found all finite regular polyhedra of index $2$ with vertices in two orbits under the symmetry group. There are twenty-two families of such polyhedra, where polyhedra are in the same family if they differ only in the relative diameters of their vertex orbits. Each of these polyhedra is orientable with a face-transitive symmetry group, but only two have planar faces. No finite regular polyhedron of index $2$ can be chiral, although Schulte~\cite{chi1,chi2} has shown that there are six classes of infinite regular polyhedra of index $2$ which are chiral.

The main result of this paper is to complete the classification of finite regular polyhedra of index $2$, regardless of orientability or planarity of faces or vertex-figures, by finding all such polyhedra with vertices in one orbit under the symmetry group. There are precisely ten such polyhedra. This paper, which, like its predecessor \cite{cs}, is based on \cite{cut}, is organized as follows. Section~\ref{mapo} sets out definitions and basic properties of maps and regular polyhedra. Section~\ref{vedgs} establishes constraints on the vertex positions and edge lengths of regular polyhedra of index $2$. In Section~\ref{fshap}, some results concerning `face shape', a concept introduced in \cite{cs}, are obtained. All regular polyhedra of index $2$ with edges of different lengths are found in Section~\ref{difel}. All index $2$ polyhedra with one vertex orbit and with edges of the same length are found in Sections~\ref{eopt}, \ref{dirty}, and \ref{bicty}. Finally, a non-Petrie duality that applies to all regular polyhedra of index $2$ is described in Section~\ref{npdua}.

\section{Maps and Polyhedra}
\label{mapo}

Coxeter-Moser~\cite{cm} defines a map $M$ as a decomposition of a closed $2$-manifold into non-overlapping, simply-connected regions by means of arcs. The regions are called the faces of $M$, the arcs are called the edges of $M$, and the intersections of the arcs are called the vertices of $M$. Thus each edge is incident with precisely two vertices, one at each end, and subtends precisely two faces. In this way, the vertices, edges, and faces of a map form a partially ordered set, which usually is identified with the underlying map. If $M$ is finite, the underlying surface will be compact.

The cyclically ordered set of edges bordering each face is called the boundary of that face. The f-vector of $M$ is the triplet $(f_{0}, f_{1}, f_{2})$, where $f_{i}$ is the number of vertices, edges, and faces, respectively. We always assume that each of these is finite. A flag of $M$ is a set consisting of one vertex, one edge incident with this vertex, and one face containing this edge. The map is said to be of type $\{p, q\}$ if each of its faces is a topological $p$-gon (so has a boundary consisting of $p$ edges), and if $q$ edges meet at each vertex.

The automorphism group of $M$ is the group of incidence preserving bijections of the partially ordered set of vertices, edges, and faces of $M$, and is denoted by $\Gamma(M)$. Following McMullen and Schulte~\cite{arp}, we say that a map $M$ is (combinatorially) regular if it is flag-transitive. Thus a regular map is reflexible in the sense of \cite{cm}. The automorphism group $\Gamma(M)$ of a regular map $M$ is generated by three involutions, $\rho_{0}$, $\rho_{1}$, and $\rho_{2}$. For a fixed flag, consisting of a vertex $F_{0}$, an edge $F_{1}$, and a face $F_{2}$, these generators can be defined by
$\rho_{i}(F_{j})=F_{j}$ if and only if $i \neq j$, for $i,j = 0, 1, 2$.

These generators satisfy the standard relations
\begin{equation}
\label{relone}
\rho_{0}^{2} = \rho_{1}^{2} = \rho_{2}^{2} =
(\rho_{0}\rho_{1})^{p} = (\rho_{1}\rho_{2})^{q} = (\rho_{0}\rho_{2})^{2} = 1,
\end{equation}
but in general satisfy other independent relations also.

The automorphism $\sigma_{1}:=\rho_{0}\rho_{1}$ cyclically permutes the successive edges on the boundary of the face $F_{2}$, and the automorphism $\sigma_{2}:=\rho_{1}\rho_{2}$ cyclically permutes the successive edges meeting at the vertex $F_{0}$ of this face. The elements $\sigma_{1}$ and $\sigma_{2}$ generate the combinatorial rotation subgroup $\Gamma^{+}(M)$ of $\Gamma(M)$, consisting of the automorphisms that can be expressed as products of an even number of generators from $\rho_{0}$, $\rho_{1}$, and $\rho_{2}$.

Every map $M$, with f-vector $(f_{0}, f_{1}, f_{2})$, has associated with it a dual map $M^{*}$ with f-vector $(f_{2}, f_{1}, f_{0})$, with one vertex of $M^{*}$ lying in each face of $M$, and vice versa, and with each edge of $M^{*}$ crossing one edge of $M$. Thus if $M$ is of type $\{p, q\}$ then $M^{*}$ is of type $\{q, p\}$. Additionally, if $M$ is regular then it also has associated with it a polygon called the Petrie polygon. Given $M$, a Petrie polygon is a `zigzag' along the edges of $M$ such that every two but no three successive edges of the polygon are edges of a single face of $M$. Because of the flag-transitivity of $\Gamma(M)$, the Petrie polygons of $M$ are all combinatorially equivalent under $\Gamma(M)$, so we may talk about the Petrie polygon of $M$.

Identifying those pairs of vertices of the regular tessellation $\{p, q\}$ on the $2$-sphere or in the Euclidean or hyperbolic plane which are separated by $r$ steps along a Petrie polygon of the tessellation gives, for suitable values of $r$, a finite regular map $M$ denoted by $\{p, q\}_{r}$. The value of $r$ gives us one further relation to complement those in (\ref{relone}),		$(\rho_{0}\rho_{1}\rho_{2})^{r} = 1$. The map $\{p, q\}_{r}$ then has Petrie polygons of length $r$.

If we leave the vertices and edges of the regular map $M$ of type $\{p, q\}_{r}$ in place, and replace the faces by the Petrie polygons, we get a new map of type $\{r, q\}_{p}$ with the same automorphism group, but not necessarily on the same surface. Thus there is one common automorphism group for six related regular maps of type $\{p, q\}_{r}$, $\{p, r\}_{q}$, $\{q, p\}_{r}$, $\{q, r\}_{p}$, $\{r, p\}_{q}$, and $\{r, q\}_{p}$.	(See \cite{cm},\cite{ms1},\cite{arp},\cite{wills}.)

For a regular map, we also have the combinatorial identities $qf_{0} = 2f_{1}$, and $pf_{2} = 2f_{1}$, which can be obtained by counting edge `ends' and edge `sides', respectively.

A polyhedron, $P$, is a faithful geometric realization in Euclidean $3$-space, $\E$, of a map, $M$ (see \cite[Ch. 5]{arp}). A face of $P$ is determined by its boundary, which is a closed simple sequence of edges of $P$. Thus the vertices, edges, and faces of $P$ are faithfully represented by points, line segments, and closed simple sequences of edges; that is, there exist bijections between the sets of vertices, edges, and faces of $M$ and the sets of vertices, edges, and faces of $P$, respectively. If $M$ is a regular map, then $P$ is called a combinatorially regular polyhedron. The boundary of a face of a combinatorially regular polyhedron need not be planar or convex.

A consequence of the definition given above of a combinatorially regular polyhedron, $P$, as a faithful realization of a regular map, $M$, through bijections between $M$ and $P$, is that no face of $P$ has a vertex occurring more than once in its boundary.

According to our definition, a polyhedron $P$ is a geometric graph in space with a distinguished family of edge cycles, the faces of $P$. The faces, as noted, need not be planar. The polyhedral realizations that we consider do not include realizations of regular maps where more than one edge joins two vertices or a vertex lies on only two edges. Specifically, we require that $p \geq 3$ and $q \geq 3$, for maps of type $\{p, q\}$, in order to achieve non-degenerate geometric realization in $\E$. A consequence of this constraint is that it is not the case for any combinatorially regular polyhedron that two faces share two consecutive edges in their boundaries.

The geometric symmetry group, $G(P)$, of $P$ consists of the isometries of $\E$ that map $P$ to itself. This is the group of rotations and (point or plane) reflections that preserve $P$, and can be viewed as a subgroup of $\Gamma(M)$ = $\Gamma(P)$, when we have identified $M$ and $P$. The index of $P$ is defined to be the index of this subgroup in $\Gamma(P)$. Thus a `regular polyhedron of index $2$' is a combinatorially regular, geometric polyhedron whose symmetry group, $G(P)$, has index $2$ in the automorphism group, $\Gamma(M)$. A consequence of this is that a regular polyhedron of index $2$ has precisely two flag orbits under $G(P)$.
An important subgroup of $G(P)$ is $G^{+}(P)$, the group of rotations that preserve $P$. If $P$ is non-orientable, the rotation group, $G^{+}(P)$, is the same as the symmetry group, $G(P)$; but if $P$ is orientable, then $G^{+}(P)$ is a subgroup of $G(P)$ of index $2$.

\section{Vertices and Edges}
\label{vedgs}

Throughout we assume that $P$ is a regular polyhedron of index $2$ whose vertices lie on one orbit under $G(P)$, the symmetry group of $P$. We take $P$ to be of type $\{p_{P}, q_{P}\}$, and to have f-vector $(f_{0}, f_{1}, f_{2})$, so that $f_{0}$, $f_{1}$, and $f_{2}$ are the number of vertices, edges, and faces, respectively, of $P$.

By Theorems 4.1 of \cite{cs} and 3.1 of \cite{cut} we know that $G(P)$ is the full symmetry group of a Platonic solid, and by Lemma 4.4 of \cite{cs} the vertices of $P$ coincide with the vertices of $S$, where $S$ is either a Platonic solid, a truncated Platonic solid (that is, an octahedron occurring as a truncated tetrahedron), a cuboctahedron, or an icosidodecahedron. Throughout $S = S(P)$ will refer to this object. We let $q_{S}$ be the number of edges of $S$ that end at any specified vertex of $S$, so that $q_{S}$ is $3$, $4$, or $5$.

Up to similarity, $P$ is metrically unique. Without loss of generality, we fix the size of $P$ by taking the edge length of $S$ to be $1$.

\begin{lemma}
\label{lem1}
Let $P$ be a regular polyhedron of index $2$ with its vertices on one orbit, and let $S$ be the convex hull of its vertices. Then the symmetry group, $G(P)$, of $P$ is the full symmetry group, $G(S)$, of $S$.
\end{lemma}

\begin{proof}
We first eliminate the case that $G(P)$ is a proper subgroup of $G(S)$, which would occur exactly when $G(P)$ is the full tetrahedral group and $S$ is the octahedron, viewed as a truncated tetrahedron. Suppose that $P$ has tetrahedral symmetry and the vertices of $P$ are on the midpoints of the edges of a tetrahedron. In this case $S$ is an octahedron; $P$ has $12$ edges, since $|G(P)| = 24$; and $q_{P} = 2f_{1} / f_{0} = 4$. Consequently, since at least three of the four edges at a vertex must be edges of $S$, the symmetry also forces the fourth edge to be an edge of $S$. Thus all edges of $P$ must coincide with those of $S$. Since $P$ is a polyhedron, adjacent edges of any face of $P$ can not also be adjacent edges of any other face of $P$, for this would give $q_{P} = 2$. It follows that in this case, adjacent edges of any face of $P$ must be adjacent edges of a face of $S$, allowing only two possibilities for the type of faces of $P$, namely triangles or skew hexagons. Thus $P$ can only be a regular octahedron or the Petrie dual of a regular octahedron, and so is not an index $2$ polyhedron with tetrahedral symmetry.

Further, if $S$ is a cube then $P$ can not have tetrahedral symmetry as that would require two vertex orbits, not one. Thus we have that if $S$ is a tetrahedron, $P$ has tetrahedral symmetry; if $S$ is an octahedron or a cube, $P$ has octahedral symmetry; and in all other cases, $P$ has icosahedral symmetry. The lemma then follows from Theorem 4.1 of \cite{cs}.
\end{proof}

\begin{lemma}
\label{lem2}
Let $P$ be a regular polyhedron of index $2$ with its vertices on one orbit, and let $S$ be the convex hull of its vertices. Then $S$ can not be a tetrahedron or an octahedron. If $S$ is a cube or a dodecahedron, $q_{P} = 6$. If $S$ is an icosahedron, $q_{P} = 10$. If $S$ is a cuboctahedron or an icosidodecahedron, $q_{P} = 4$.
\end{lemma}

\begin{proof}
Since the symmetry group of $P$ is $G(S)$, it follows that $f_{1}$, the number of edges of $P$, is $12$ if $P$ has tetrahedral symmetry, $24$ if $P$ has octahedral symmetry, and $60$ if $P$ has icosahedral symmetry. Since $P$ has the same number of vertices as $S$, and $q_{P} = 2f_{1} / f_{0}$, we have $q_{P} = 4$ if $S$ is a cuboctahedron or an icosidodecahedron, and $q_{P} = 2q_{S}$ if $S$ is a Platonic solid. But if $S$ is a tetrahedron or an octahedron, $2q_{S} = f_{0} + 2$, which contradicts $q_{P} < f_{0}$.
\end{proof}

As in \cite{cs}, we define the length of any edge $\{v_{1},v_{2}\}$ of $P$ to be the length of the shortest path from $v_{1}$ to $v_{2}$ along edges of $S$. The length of any edge of $P$ is thus an integer. This concept of length will be modified later when $S$ is a cuboctahedron or icosidodecahedron, but for now we adopt the above definition.

\begin{lemma}
\label{lem3}
Let $P$ and $S$ be as defined above. If $P$ has edges of different lengths, there are three possible combinations of $S$ and edge lengths of $P$. These are: a) $S$ is a cube and $P$ has edges of length $1$ and $2$; b) $S$ is a dodecahedron and $P$ has edges of length $1$ and $4$; or c) $S$ is an icosahedron and $P$ has edges of length $1$ and $2$. If all edges of $P$ are the same length, then either d) $S$ is a dodecahedron and $P$ has edges of length $2$ or $3$; or e) $S$ is a cuboctahedron or an icosidodecahedron.
\end{lemma}

\begin{proof}
By Lemma 6.1 of \cite{cs}, no edge of $P$ goes between opposite vertices of $S$. Thus the longest possible edge lengths of $P$ are $2$ if $S$ is a cube, an icosahedron or a cuboctahedron; and $4$ if $S$ is a dodecahedron, or an icosidodecahedron.

If $S$ is an icosahedron, $q_{P} = 10$, but for any given edge length there are at most $5$ vertices equidistant from any specified vertex of $P$. If $S$ is a cube or dodecahedron, $q_{P} = 6$, and there are at most $3$ vertices equidistant from any specified vertex, $v$, of $P$, unless $S$ is a dodecahedron and the edges of $P$ have length $2$ or $3$, in which case the stabilizer of $v$ in $G(S)$ has order $6$. Thus if $S$ is a Platonic solid then all edges of $P$ are the same length if and only if $S$ is a dodecahedron and the edge lengths of $P$ are either $2$ or $3$.

Suppose now that $S$ is a cuboctahedron or an icosidodecahedron, so that $q_{P} = 4$ by Lemma~\ref{lem2}. Let $u$ be a specified vertex of $P$, and let $u^{*}$ be the vertex opposite $u$. Since the stabilizer of $u$ in $G(S)$ has order $4$, the orbit of a vertex $v$ (distinct from $u, u^{*}$) under this stabilizer contains $4$ vertices, except when $v$ lies on the perpendicular bisector of $u$ and $u^{*}$, when it contains only $2$ vertices. Since $q_{P} = 4$ and there are only two vertices of the exceptional kind, $u$ must be connected by an edge of $P$ to a vertex $v$ not on the perpendicular bisector of $u$ and $u^{*}$, and thus there must be $4$ edges at $u$ equivalent to $\{u,v\}$. Thus if $S$ is a cuboctahedron or an icosidodecahedron then all edges of $P$ are the same length.
\end{proof}

\section{Face Shapes}
\label{fshap}

The notion of face shapes was introduced in \cite{cs}. If the edge lengths of a regular polyhedron, $P$, of index $2$ are known, then any face, $F$, of $P$ can be specified by four symbols $[a,b,c,d]$, provided that a directed starting edge of $P$ has been specified (modulo $G^{+}(S)$). The symbols represent the change of direction at successive vertices of the boundary of $F$ when it is projected onto the circumsphere of $S$. These directions in counterclockwise order are denoted by $r, f,$ and $l$ (`right', `forward', `left'), when there are three, and $hr, sr, (f,) sl,$ and $hl$ (`hard right', `soft right', etc.), when there are four or five. We say that $[a,b,c,d]$ is the shape of $F$. A fuller description, with examples, is provided in Section 6 of \cite{cs}.

In general $[a,b,c,d]$ is not a unique designation for $F$. $F$ may also be described by $[b,c,d,a]$, $[c,d,a,b]$, or $[d,a,b,c]$, depending on the number of directed edge orbits of $P$ under $G^{+}(P)$. In addition, if we traverse the boundary of $F$ in the opposite direction we see that the four shapes above may be the same as $[d',c',b',a']$, $[a',d',c',b']$, $[b',a',d',c']$, and $[c',b',a',d']$, where $x'$ represents the direction $x$ when traversing the boundary in the opposite direction.

In addition to being a convenient notation, we can also on occasion compare two face shapes to obtain further information. The stipulation that the starting edges for the shapes of faces of $P$ all lie in the same edge orbit of directed edges under $G^{+}(S)$ is necessary to properly compare faces of $P$. Specifically, if two faces of $P$, say $F_{1}$ and $F_{2}$, lie in the same orbit under $G^{+}(S)$, then they will have the same shape. We can obtain constraints on this common shape in the following manner. Suppose the shape of $F_{1}$ is $[a,b,c,d]$. Then the shape of $F_{1}$ can also be represented by up to $8$ derivations of this shape, namely the cyclic permutations of $[a,b,c,d]$, as well as those of $[d',c',b',a']$. No other quartet of symbols can possibly represent the shape of $F_{1}$. Therefore, since $F_{2}$ has the same shape as $F_{1}$, one of these eight derivations must be the shape of $F_{2}$, with symbol-by-symbol equality. This observation is used in the proof of Lemma~\ref{lem8}.

Let $F$ be a face of $P$. $P$ has $2|G^{+}(S)|/p$ faces, and either all of them or half of them are in the same orbit as $F$ under $G^{+}(S)$ and thus have the same shape as $F$. So the stabilizer, $G_{F}^{+}(P)$, of $F$ in $G^{+}(S)$, which is the subgroup of rotational symmetries of $F$ contained in $G^{+}(S)$ has order either $p/2$ (when all faces of $P$ are in one orbit under $G^{+}(S)$, in which case $p$ is even) or $p$ (when $P$ has two face orbits under $G^{+}(S)$), as stated in Lemma 5.3 of \cite{cs}.

Since $G_{F}^{+}(P)$ is a subgroup of $\Gamma_{F}(P)$, and $\Gamma_{F}(P)$ is a dihedral group generated by $\rho_{0}(F)$ and $\rho_{1}(F)$, it follows that each symmetry in $G_{F}^{+}(P)$ is a combinatorial automorphism in $\Gamma_{F}(P)$ either of the form $\sigma_{1}^{j}(F)$, where $\sigma_{1}(F) := \rho_{0}(F)\rho_{1}(F)$ and $j = 0,\ldots,p-1$, or $\rho_{v}(F)$ or $\rho_{e}(F)$, where $\rho_{v}(F)$ flips about a vertex $v$ of $F$ (while preserving $F$) and $\rho_{e}(F)$ flips the ends of an edge $e$ of $F$ (while preserving $F$).

The possible shapes of faces of $P$ are described in the following lemma.

\begin{lemma}
\label{lem4}
Suppose $F$ is a face of $P$, a regular polyhedron of index $2$.\\
If $P$ has two face orbits under $G^{+}(P)$, then either
\begin{itemize}
\item[(a)] $G_{F}^{+}(P)$ has $p$ automorphisms of the form $\sigma_{1}^{j}(F)$, and the shape of $F$ is of the form $[a,a,a,a]$ and edges of $F$ are the same length; or
\item[(b)] $G_{F}^{+}(P)$ has $p/2$ automorphisms of the form $\sigma_{1}^{2j}(F)$ and $p/2$ automorphisms of the form $\rho_{e}(F)$, and the shape of $F$ is of the form $[a,a',a,a']$ and alternate edges of $F$ are the same length; or
\item[(c)] $G_{F}^{+}(P)$ has $p/2$ automorphisms of the form $\sigma_{1}^{2j}(F)$ and $p/2$ automorphisms of the form $\rho_{v}(F)$, and the shape of $F$ is of the form $[f,f,f,f]$ and edges of $F$ are the same length, and $S$ is a cuboctahedron or an icosidodecahedron.
\end{itemize}
If $P$ has one face orbit under $G^{+}(P)$, then either
\begin{itemize}
\item[(d)] $G_{F}^{+}(P)$ has $p/2$ automorphisms of the form $\sigma_{1}^{2j}(F)$, and the shape of $F$ is of the form $[a,b,a,b]$ and alternate edges of $F$ are the same length; or
\item[(e)] $G_{F}^{+}(P)$ has $p/4$ automorphisms of the form $\sigma_{1}^{4j}(F)$ and $p/4$ automorphisms of the form $\rho_{e}(F)$, and the shape of $F$ is of the form $[a,a',b,b']$ or $[a,b,b',a']$ and alternate edges of $F$ are the same length; or
\item[(f)] $G_{F}^{+}(P)$ has $p/4$ automorphisms of the form $\sigma_{1}^{4j}(F)$ and $p/4$ automorphisms of the form $\rho_{v}(F)$, and the shape of $F$ is of the form $[a,f,a',f]$ or $[f,a,f,a']$ and edges of $F$ are the same length, and $S$ is a cuboctahedron or an icosidodecahedron.
\end{itemize}
\end{lemma}

\begin{proof}
Let $e'$ (respectively $v'$) be an edge (resp. vertex) of $F$ two steps away from $e$ (resp. $v$) along the boundary of $F$. Since $\sigma_{1}^{4}(F) \in G_{F}^{+}(P)$ (see \cite{cs}, Section 6), and $\rho_{e}(F)\rho_{e'}(F) = \sigma_{1}^{4}(F) = \rho_{v}(F)\rho_{v'}(F)$, we have that $\rho_{e'}(F) \in G_{F}^{+}(P)$ if and only if $\rho_{e}(F) \in G_{F}^{+}(P)$, and $\rho_{v'}(F) \in G_{F}^{+}(P)$ if and only if $\rho_{v}(F) \in G_{F}^{+}(P)$. We deduce from this that $G_{F}^{+}(P)$ can not contain involutions of the form $\rho_{e}(F)$ as well as of the form $\rho_{v}(F)$, for if so then $G_{F}^{+}(P)$ must also contain a pair $\rho_{e}(F)$, $\rho_{v}(F)$, where $v$ is a vertex of $e$. In this case $\sigma_{1}(F) = \rho_{e}(F)\rho_{v}(F) \in G_{F}^{+}(P)$ and, together with $\rho_{e}(F)$, generates a group of order $2p$, contradicting the fact that $G_{F}^{+}(P)$ has order $p/2$ or $p$.

Since $\rho_{e}(F)$ is a half-turn that flips the two vertices of $F$ in $e$, it rotates $P$ about an axis through {\em o\/} (the center of P) and the midpoint of $e$. If $\rho_{e}(F)$ is in $G_{F}^{+}(P)$ this axis must also pass through the midpoint of the edge of $F$ opposite $e$; if it passed through a vertex, $v$, of $F$, $\rho_{v}(F)$ would also be in $G_{F}^{+}(P)$. An analogous argument applies to $\rho_{v}(F)$. Thus if $G_{F}^{+}(P)$ does not contain $\sigma_{1}(F)$, $p$ must be even.

Since $\sigma_{1}^{4}(F)$ and $\sigma_{1}^{p}(F)$ are in $G_{F}^{+}(P)$, we also have that $\sigma_{1}^{k}(F) \in G_{F}^{+}(P)$, where $k = hcf(4, p) = 1, 2,$ or $4$. It follows that the automorphisms in $G_{F}^{+}(P)$ are as described in each of the six cases above. It remains to determine the face shape and edge lengths for these cases.

Clearly, if $\sigma_{1}(F) \in G_{F}^{+}(P)$, then all edges of $F$ have the same length, and all four symbols in the face shape are the same, and so $F$ has shape $[a,a,a,a]$. Similarly, if $\sigma_{1}^{2}(F) \in G_{F}^{+}(P)$, then $a$ and $c$ are the same, as are $b$ and $d$, and so $F$ has shape $[a,b,a,b]$, and alternate edges of $F$ are the same length.

Suppose $G_{F}^{+}(P)$ contains an automorphism of the form $\rho_{v}(F)$. Then $\rho_{v}(F)$, being an involution, must be a half-turn about an axis passing through a vertex, $v$, of $S$. Thus $q_{S}$ is even, so $S$ must be a cuboctahedron or an icosidodecahedron, and by Lemma~\ref{lem3} all edges of $P$ are the same length. The edges of $F$ that contain $v$ must span a plane that contains the rotation axis of $\rho_{v}(F)$. Hence, the boundary of $F$ does not change direction at $v$, indicated by an `$f$' at the corresponding position in the shape of $F$. Moreover, we can apply the same analysis to $\rho_{v'}(F)$, where $v'$ is as defined above. It follows that the face shape also has an `$f$' at the position corresponding to $v'$. Finally, the two other symbols in the face shape are of the form $a$ and $a'$, since they represent the vertices of $F$, distinct from $v$, that are the endpoints of the edges of $F$ at $v$, and these are mapped into each other by the half-turn $\rho_{v}(F)$. Thus $F$ has shape $[a,f,a',f]$ or $[f,a,f,a']$.

Finally, suppose $G_{F}^{+}(P)$ contains an automorphism of the form $\rho_{e}(F)$. Then, for $e'$ as defined above, we also have $\rho_{e'}(F) \in G_{F}^{+}(P)$. Both $\rho_{e}(F)$ and $\rho_{e'}(F)$ must be half-turns. Since $\rho_{e}(F)$ is a half-turn that flips the two vertices of $F$ in $e$, it follows that alternate edges of $F$ are the same length, and the symbols of the face shape corresponding to those vertices are adjacent and of the form $a$, $a'$. We argue similarly for $\rho_{e'}(F)$ and the two vertices of $e'$, giving entries $b$, $b'$ in the face shape. Thus the shape of $F$ is $[a,a',b,b']$ or $[a,b,b',a']$.
\end{proof}

We will see later that cases $(e)$ and $(f)$ do not exist, so that it will always be the case that $G_{F}^{+}(P)$ contains $\sigma_{1}^{2}(F)$.

\begin{lemma}
\label{lem5} Suppose $P$, a regular polyhedron of index $2$, has one face orbit under $G(P)$. Then no face of $P$ has shape $[f,f,f,f]$.
\end{lemma}

\begin{proof}
Let the vertices of $P$ be coincident with the vertices of $S$, where by Lemma~\ref{lem3} $S$ is a Platonic solid, a cuboctahedron, or an icosidodecahedron. Let $F$ be a face of $P$, and assume that $F$ has shape $[f,f,f,f]$. If $P$ has one face orbit under $G^{+}(P)$, then each face of $P$ also has shape $[f,f,f,f]$. If $P$ has two face orbits under $G^{+}(P)$, then each face in the same orbit as $F$ has shape $[f,f,f,f]$, and each face not in that orbit is equivalent under $G^{+}(P)$ to a planar reflection of $F$ (since $P$ has one face orbit under $G(P)$), and so has shape $[f',f',f',f']$. It is straightforward to verify that in all cases $f' = f$, so that $[f',f',f',f']$ is the same shape as $[f,f,f,f]$. (In fact $F$ is planar, and is coplanar with \emph{o}, except when $S$ is a dodecahedron and the edges of $P$ are all the same length.) Thus all faces of $P$ have shape $[f,f,f,f]$. Let $F'$ be a face of $P$ adjacent to $F$, and let $\{u,v\}$ be their common edge. Both $F$ and $F'$ have a change of direction of $f$ ($= f'$) at $v$, and so both share another common edge $\{v,w\}$. But this is not possible, as $q_{P} \geq 3$.
\end{proof}

We establish for later use two results concerning the face shapes of Petrie duals of polyhedra.

\begin{lemma}
\label{lem6} Let $P$ be a finite regular polyhedron of index $2$. If each face of $P$ has shape either $[a,a',a,a']$ or $[a',a,a',a]$, where $a$ represents any specific change of direction, then the shape of each Petrie polygon of $P$ is either $[a,a,a,a]$ or $[a',a',a',a']$. If each face of $P$ has shape either $[a,a,a,a]$ or $[a',a',a',a']$, then the shape of each Petrie polygon of $P$ is either $[a,a',a,a']$ or $[a',a,a',a]$.
\end{lemma}

\begin{proof}
Let $\{u,v\}$ be any edge of $P$, and let $F_{1}$ and $F_{2}$ be the two faces of $P$ bordering $\{u,v\}$. Let the boundary of $F_{1}$ be $\{u,v,w_{1},\ldots\}$, and the boundary of $F_{2}$ be $\{u,v,w_{2},\ldots\}$. The boundaries of $F_{1}$ and $F_{2}$ each have a change of direction represented by either $a$ or $a'$ at vertex $v$. If these two changes of direction were the same, then $w_{1}$ and $w_{2}$ would coincide since the edges $\{v,w_{1}\}$ and $\{v,w_{2}\}$ have the same length, and so the boundaries of $F_{1}$ and $F_{2}$ would share two consecutive edges, and $P$ could not be a polyhedron. Thus at vertex $v$, the boundary of either $F_{1}$ or $F_{2}$ has a change of direction represented by $a$, and the boundary of the other one has a change of direction represented by $a'$. (Note that this implies $a \neq f$.) Thus the neighboring faces along any edge of $P$ have different changes of direction at an end-point of that edge. The lemma follows from the way that Petrie polygons are constructed.
\end{proof}

\begin{lemma}
\label{lem7} If $P$ is a regular polyhedron of index $2$ with all its faces having shape $[a,f,b,f]$ or $[f,c,f,d]$, where $a, b, c,$ and $d$ represent any change of direction (including $f$), then the Petrie dual $Q$ of $P$ is a regular polyhedron of index $2$ with two face orbits under $G(P)$, and $Q$ has a face with shape $[f,f,f,f]$.
\end{lemma}

\begin{proof}
Assume that $Q$ is a polyhedron. Then Lemma 3.2 of \cite{cs} tells us that $Q$ is a regular polyhedron of index $2$ with $G(Q) = G(P)$. We establish the existence of a face of $Q$ of shape $[f,f,f,f]$, ignoring for the moment the question of polytopality of $Q$. Let $\{u,v\}$ be any edge of $P$, and let $F_{1}$ and $F_{2}$ be the two faces of $P$ bordering $\{u,v\}$. The boundaries of $F_{1}$ and $F_{2}$ each have a change of direction of $f$ at either $u$ or $v$. If it were at the same vertex, then the boundaries of $F_{1}$ and $F_{2}$ would share two consecutive edges, and $P$ could not be a polyhedron. So the boundary of either $F_{1}$ or $F_{2}$ has an $f$ at $u$, and the boundary of the other one has an $f$ at $v$. It follows from the way the Petrie dual is constructed that one of the faces of $Q$ has shape $[f,f,f,f]$. This is true regardless of whether or not $Q$ is polytopal. If $Q$ is not polytopal, we simply regard $Q$ as the collection of Petrie polygons of $P$, and their vertices and edges.

By Lemma 3.2 of \cite{cs}, to prove that $Q$ is a polyhedron it suffices to show that a Petrie polygon of $P$ does not revisit a vertex. Since $P$ is combinatorially regular, it is enough to verify this condition for a single Petrie polygon, and we will take the polygon which is a face of shape $[f,f,f,f]$ of $Q$. Regardless of the choice of $S$ (a dodecahedron, cuboctahedron, or icosidodecahedron), inspection of the few possibilities shows that no polygon of shape $[f,f,f,f]$ and of any given edge length does revisit a vertex. Hence the face of $Q$ with shape $[f,f,f,f]$ does not revisit a vertex either. It follows that $Q$ is a polyhedron, and so by Lemma~\ref{lem5} it has two face orbits under $G(P)$.
\end{proof}

\section{Different Edge Lengths}
\label{difel}

If $P$ has edges of different lengths, there are just three possible combinations of $S$ and edge lengths of $P$, by Lemma~\ref{lem3}. These are: a) $S$ is a cube and $P$ has edges of length $1$ and $2$; b) $S$ is a dodecahedron and $P$ has edges of length $1$ and $4$; or c) $S$ is an icosahedron and $P$ has edges of length $1$ and $2$. In particular, $S$ is a Platonic solid.

We introduce a further adaptation of the face shape notation. If the vertex positions and edge lengths of $P$ are clear, as well as the convention for identifying the starting edge of the face shape, then $P$ can be characterized in terms of the shape of its faces. The notation used here is that if $P$ has one face orbit under $G^{+}(P)$ with the faces of $P$ having shape $[a,b,c,d]$ then $P$ is said to have shape $[a,b,c,d]$, whereas if $P$ has two face orbits under $G^{+}(P)$ with the faces in one orbit having shape $[a,b,c,d]$, and the faces in the other orbit having shape $[e,g,h,i]$, then $P$ is said to have shape $[a,b,c,d] \& [e,g,h,i]$ (or, equivalently, $[e,g,h,i] \& [a,b,c,d]$).

\begin{lemma}
\label{lem8} If $P$ has edges of different lengths, then $P$ has one of the four shapes $[a,a,a,a]$, $[a,a,a',a']$, $[a,a',a',a]$, or $[a,a',a,a'] \& [a',a,a',a]$. In each case $a \neq f$, and $G(P)$ acts face-transitively.
\end{lemma}

\begin{proof}
Let $F$ be a face of $P$ such that the boundary of $F$ contains edges of different lengths. From Lemma~\ref{lem4}, we have that $F$ has alternating edge lengths, and $G_{F}^{+}(P)$ can only contain $\sigma_{1}^{2}(F)$, or automorphisms of the form $\rho_{e}(F)$, or both.

Let $F'$ be the image of $F$ under any plane reflection of $P$, so that $F'$ and $F$ are in the same orbit under $G(P)$. Then the boundaries of both $F$ and $F'$ have alternating edge lengths. If the shape of $F$ is $[a,b,c,d]$ then the shape of $F'$ is one of  $[d,c,b,a]$, $[b,a,d,c]$, $[a',b',c',d']$, or $[c',d',a',b']$, because the starting edges of the shapes of $F$ and $F'$ lie in the same directed edge orbit under $G^{+}(P)$, and so have the same length. As noted previously, if $F$ and $F'$ are in the same orbit under $G^{+}(P)$ then their shapes will be identical.

If $G_{F}^{+}(P)$ contains both $\sigma_{1}^{2}(F)$ and an automorphism of the form $\rho_{e}(F)$, then $G_{F}^{+}(P)$ has order $p$, so $P$ has two face orbits under $G^{+}(P)$, and the shape of $F$ is $[a,a',a,a']$ (see Lemma~\ref{lem4}). Consequently, the shape of $F'$ is $[a',a,a',a]$. If $F$ and $F'$ are in the same orbit under $G^{+}(P)$, then they have the same shape; in that case, equating the shapes of $F$ and $F'$ gives a shape of $[a,a,a,a]$, where $a' = a$. But there is only one such face (which is $[f,f,f,f]$) and so $P$ would have only one face orbit under $G^{+}(P)$, not two. Therefore $F$ and $F'$ are in different orbits under $G^{+}(P)$, and it follows that $P$ has shape $[a,a',a,a'] \& [a',a,a',a]$ with $a \neq f$, and has only one face orbit under $G(P)$.

If $G_{F}^{+}(P)$ contains $\sigma_{1}^{2}(F)$, but no automorphisms of the form $\rho_{e}(F)$, then $G_{F}^{+}(P)$ has order $p/2$, so $P$ has one face orbit under $G^{+}(P)$ and $G(P)$, and the shape of $F$ is $[a,b,a,b]$. Consequently, the shape of $F'$ must be $[b,a,b,a]$ or $[a',b',a',b']$. Since $F$ and $F'$ are in the same orbit under $G^{+}(P)$, they have the same shape. Equating the shapes of $F$ and $F'$ gives just two possibilities: either $a = b$; or $a = a'$ and $b = b'$ (and hence $a = b = f$). The latter gives the shape $[f,f,f,f]$ for $F$, which is excluded by Lemma~\ref{lem5}. Hence $F$ must have shape $[a,a,a,a]$, with $a \neq f$.

If $G_{F}^{+}(P)$ contains only automorphisms of the form $\rho_{e}(F)$, then $G_{F}^{+}(P)$ has order $p/2$, so $P$ has one face orbit under $G^{+}(P)$, and the shape of $F$ is either $[a,a',b,b']$ or $[a,b,b',a']$ (depending on whether or not the edge flipped by $\rho_{e}(F)$ is the same length as the starting edge). It follows that the shape of $F'$ is either $[b',b,a',a]$ or $[a',a,b',b]$ in the first instance, and either $[a',b',b,a]$ or $[b,a,a',b']$ in the second instance. Since $F$ and $F'$ are in the same orbit under $G^{+}(P)$, they have the same shape, which we find by equating the shapes of $F$ and $F'$, and which must be $[a,a',a',a]$ or $[f,f,f,f]$ in the first case, and $[f,f,f,f]$ or $[a,a,a',a']$ in the second case. However, by Lemma~\ref{lem5}, if $P$ has one face orbit under $G(P)$ then the face shape can not be $[f,f,f,f]$. Thus $F$ must have shape $[a,a',a',a]$ or $[a,a,a',a']$, with $a \neq f$.
\end{proof}

With this result we obtain all possibilities for finite regular polyhedra of index $2$ with different edge lengths.

\begin{theorem}
\label{thm1}
There are precisely two combinatorially regular polyhedra of index $2$ with vertices on one orbit and edges of different lengths. They are Petrie duals of each other, are of type $\{6, 6\}_{6}$, and have their vertices at the vertices of a dodecahedron. One is orientable, of genus $11$, with planar faces; the other is non-orientable, of genus $22$, with non-planar faces.
\end{theorem}

\begin{proof}
If $S$ is a cube (and $P$ has edges of length $1$ and $2$) or a dodecahedron (and $P$ has edges of length $1$ and $4$), then $q_{P} = 6$, by Lemma~\ref{lem2}, and so there are three possible directions ($r, f,$ or $l$) for a face boundary to continue from any vertex, since edge lengths alternate along a face and around a vertex. By Lemma~\ref{lem8}, there are seven possible shapes for $P$. Let $F$ be a face of $P$, and $F'$ be the image of $F$ under any plane reflection of $P$. If $G_{F}^{+}(P)$ does not contain both $\sigma_{1}^{2}(F)$ and an automorphism of the form $\rho_{e}(F)$, then $F$ and $F'$ are in the same orbit under $G^{+}(P)$, so that $[r,r,r,r]$, $[r,r,l,l]$, and $[r,l,l,r]$ represent the same faces as $[l,l,l,l]$, $[l,l,r,r]$, and $[l,r,r,l]$, respectively. (Note here that each edge of $P$ lies in a reflection plane of $S$.) Eliminating these duplications from separate consideration reduces the analysis to polyhedra of shapes $[r,r,r,r]$, $[r,r,l,l]$, $[r,l,l,r]$, or $[r,l,r,l] \& [l,r,l,r]$.

If $S$ is an icosahedron (and $P$ has edges of length $1$ and $2$), then $q_{P} = 10$, by Lemma~\ref{lem2}, and so there are five possible directions for a face boundary to continue from any vertex. By Lemma~\ref{lem8}, there are then $14$ possible shapes for $P$, and in the same manner as above, we see that $6$ pairs of these represent duplicate faces, so that we are left with $8$ possible different shapes for $P$, namely $[hr,hr,hr,hr]$, $[hr,hr,hl,hl]$, $[hr,hl,hl,hr]$, and $[hr,hl,hr,hl] \& [hl,hr,hl,hr]$, as well as $[sr,sr,sr,sr]$, $[sr,sr,sl,sl]$, $[sr,sl,sl,sr]$, and $[sr,sl,sr,sl] \& [sl,sr,sl,sr]$.

These sixteen structures, four each when $S$ is a cube or a dodecahedron and eight when $S$ is an icosahedron, are thus the only candidates for finite regular polyhedra of index $2$ with edges of different length. Since $P$ is a regular polyhedron, it is, by definition, a faithful realization of a regular map, $M$, through bijections between $M$ and $P$, and thus we have that no face of $P$ has a vertex occurring more than once in its boundary.

Figure \ref{fig1}, below, shows by direct construction (using the convention that the starting edge in the face shape notation is an edge of length $1$) a portion of a face for seven of the sixteen structures. In each case the faces self-intersect, so these structures can be eliminated immediately. Three of the diagrams each show a portion of a face for two structures, since in each case the faces of the two structures share the portion shown.

\begin{figure} [ht]
\centering
\includegraphics[trim = 20mm 40mm 25mm 30mm, clip, scale = 0.6]{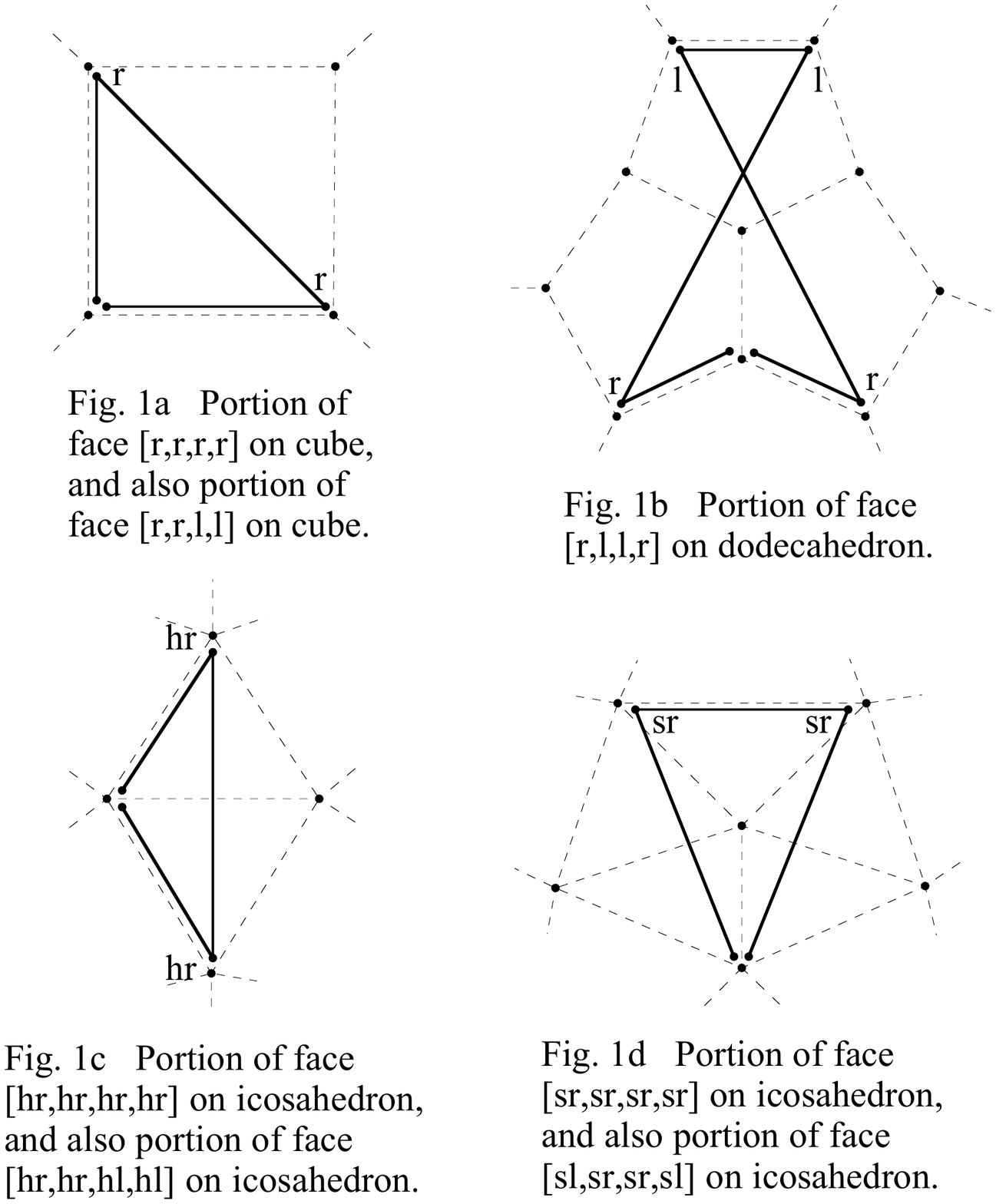}\\
\caption{}
\label{fig1}
\end{figure}

The edges in Figure \ref{fig1} have been shrunk slightly to show the underlying structure of $S$ more clearly. Note that the face $[sl,sr,sr,sl]$ in Figure 1d is the face $[sr,sl,sl,sr]$ with the same starting edge, but traversed in the opposite direction.

None of the diagrams in Figure \ref{fig1} represents the entire face, as they each have an odd number of edges, whereas any face with edges of different lengths must have an even number of edges. Thus each of the seven faces that are partially shown in Figure \ref{fig1} self-intersects at a vertex, so the structures containing those faces are not regular polyhedra by definition. This eliminates from consideration $7$ of the $16$ candidate structures.

The seven structures that are Petrie dual to the structures in Figure \ref{fig1} are, respectively, the structures $[r,l,r,l] \& [l,r,l,r]$, and $[r,l,l,r]$ when $S$ is the cube; $[r,r,l,l]$ when $S$ is the dodecahedron; and $[hr,hl,hr,hl] \& [hl,hr,hl,hr]$, $[hr,hl,hl,hr]$, $[sr,sl,sr,sl] \& [sl,sr,sl,sr]$, and $[sr,sr,sl,sl]$ when $S$ is the icosahedron. It is the case that none of these seven structures is a regular polyhedron. This can be confirmed directly from the geometry determined by their shape. Three of these seven structures, namely $[r,r,l,l]$ when $S$ is a dodecahedron, and $[hr,hl,hl,hr]$ and $[sr,sr,sl,sl]$ when $S$ is an icosahedron, have self-intersecting faces. Diagrams, similar to those in Figure \ref{fig1}, demonstrating this are in Appendix 3 of \cite{cut}.

It is also the case that for none of the fourteen structures --- the seven in Figure \ref{fig1}, and the seven that are Petrie dual to them --- is the combinatorial automorphism $\rho_{1}$ well-defined (that is, the underlying map is not combinatorially regular). Appendix 3 of \cite{cut} contains diagrams demonstrating this for the remaining four structures that have not been shown to have self-intersecting faces, namely $[r,l,l,r]$ and $[r,l,r,l] \& [l,r,l,r]$ when $S$ is a cube, and $[hr,hl,hr,hl] \& [hl,hr,hl,hr]$ and $[sr,sl,sr,sl] \& [sl,sr,sl,sr]$ when $S$ is an icosahedron.

Thus we are left with just two possible polyhedra of index $2$ with different edge lengths, namely the two structures with shapes $[r,r,r,r]$, and $[r,l,r,l] \& [l,r,l,r]$, where $S$ is a dodecahedron. It is the case that these structures are both regular polyhedra of index $2$, and the verification of their combinatorial regularity is given in Appendix 2 of \cite{cut}. Figure \ref{fig2} shows their representative faces. These then are the only two such polyhedra with edges of different lengths. These polyhedra are derived, respectively, from cases $(d)$ and $(b)$ of Lemma~\ref{lem4}, and so for every face $F$, $G_{F}^{+}(P)$ contains $\sigma_{1}^{2}(F)$.

Each of these two polyhedra have type $\{6, 6\}$ (see Lemma~\ref{lem2}), an f-vector of $(20, 60, 20)$, and edge lengths of $1$ and $4$. The polyhedron with shape $[r,r,r,r]$ is orientable, with one face orbit under $G^{+}(P)$ (and hence also $G(P)$), and has planar faces; the polyhedron with shape $[r,l,r,l] \& [l,r,l,r]$ is non-orientable, with non-planar faces, and has two face orbits under $G^{+}(P)$, but only one face orbit under $G(P)$, so that faces in different orbits are mirror reflections of each other.
\end{proof}

\begin{figure} [ht]
\centering
\includegraphics[trim = 30mm 105mm 30mm 85mm, clip, scale = 0.6]{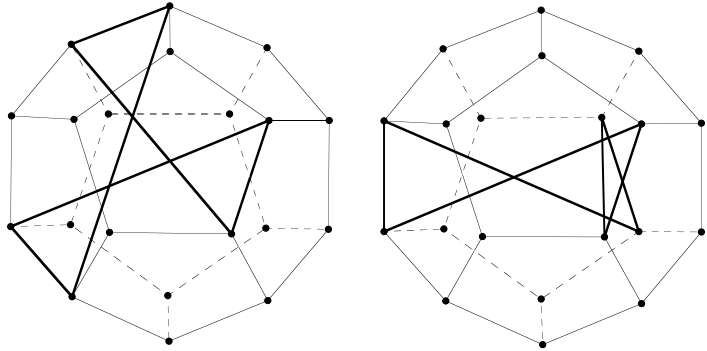}\\
\caption{The two polyhedra with edges of unequal length. They have type $\{6, 6\}_{6}$ and the vertices coincide with those of a dodecahedron. They are Petrie-dual to each other. The left one has shape $[r,r,r,r]$ and is orientable; the right one has shape $[r,l,r,l]\&[l,r,l,r]$ and is non-orientable with one face orbit under $G(P)$. Shown is one face.}
\label{fig2}
\end{figure}

\section{Edge Orientation and Polyhedral Type}
\label{eopt}

From now on we suppose that all edges of $P$ have the same length. Recall that the vertices of $P$ lie at the vertices of $S$, and are in one orbit under $G(P)$. By Lemma~\ref{lem3}, the edges of $P$ are the same length if and only if $S$ is a dodecahedron and the edge length is $2$ or $3$, or $S$ is a cuboctahedron or an icosidodecahedron.

If $S$ is a Platonic solid, any pair of vertices of $P$ at a given distance (defined as the shortest distance along edges of $S$) is equivalent under $G(P)$ to any other pair of vertices at this distance. However, when $S$ is a cuboctahedron or an icosidodecahedron this property does not hold, and accordingly we modify the definition of edge length of $P$ in order to maintain it. Non-equivalent edges of the same length exist in these cases because the cuboctahedron and icosidodecahedron each have two different face shapes. So when $S$ is a cuboctahedron or an icosidodecahedron, we modify the measure of the length of an edge of $P$ by taking the shortest path along diagonals of faces of $S$ (where a diagonal of a triangle is the same as an edge of the triangle). We keep the side of each triangle to be of length $1$, and designate as $d$ the length of the diagonal of the square in the case of the cuboctahedron (equal to $\sqrt{2}$), or the pentagon in the case of the icosidodecahedron (equal to the golden ratio, $(1+\sqrt{5})/2$). Then we can uniquely specify, up to the action of $G(P)$, any edge of $P$ by giving its length as $m+nd$, where $m$ and $n$ are integers and $m+nd$ is minimized.

\begin{lemma}
\label{lem9} Let $P$ be a regular polyhedron of index $2$ such that $P$ has one vertex orbit and all edges of $P$ have the same length. If $S$ is the dodecahedron, then the edge length of $P$ is $2$ or $3$. If $S$ is the cuboctahedron, then the edge length of $P$ is $1$ or $2$. If $S$ is the icosidodecahedron, then the edge length of $P$ is $1, 2, 3, 4, d,$ or $2d$. In all cases, $G(P)$ is transitive on the edges of $P$.
\end{lemma}

\begin{proof}
Since there are no edges between opposite vertices (by Lemma 6.1 of \cite{cs}), then by Lemma~\ref{lem3} the only possible edge lengths for $P$, other than those specified in Lemma~\ref{lem9}, are $d$ when $S$ is the cuboctahedron, and $1+d$ when $S$ is the icosidodecahedron. If $S$ is a cuboctahedron and $P$ has an edge length of $d$, then every face of $P$ would be a square, coplanar with the center of $S$, necessitating $q_{P} = 2$, which is impossible. If $S$ is an icosidodecahedron and $P$ has an edge length of $1+d$, then the vertices of $S$ are coincident with the vertices of five regular octahedra, no two of which share a vertex. Any edge of $P$, of length $1+d$, from a specified vertex $v$, will have its endpoint coincide with a vertex on the octahedron that contains $v$. $P$ can thus be a compound, but not a polyhedron. Thus the only possible edge length combinations are those specified in Lemma~\ref{lem9}. To confirm that $G(P)$ is edge transitive, we observe that this modified definition of edge length of $P$ maintains the property that all vertices equidistant from a given vertex, $v$, are equivalent under the vertex stabilizer of $v$. (Note that this would not be true if $P$ could have edge length $1+d$ on the vertices of an icosidodecahedron.)
\end{proof}

We now consider the action of $G^{+}(P) = G^{+}(S)$ on edges of $P$. As noted above, any pair of vertices of $S$ separated by a given distance (the graph distance when $S$ is a Platonic solid, or as defined above when $S$ is the cuboctahedron or icosidodecahedron) is equivalent under $G(P)$ = $G(S)$ to any other pair of vertices the same distance apart. However, this no longer remains true for equivalence under the rotation subgroup $G^{+}(P)$. Indeed, it fails precisely when $S$ is a dodecahedron and the edge length is $2$ or $3$, or when $S$ is a cuboctahedron or an icosidodecahedron. These are exactly the cases for which all edges of $P$ have the same length.

Recall that $P$ has $|\Gamma(P)|/4 = |G^{+}(P)|$ edges. The length of an orbit of an edge, $e$, under $G^{+}(P)$ equals the index (in $G^{+}(P)$) of the stabilizer of $e$ in $G^{+}(P)$. If a nontrivial element $\psi \in G(P)$ stabilizes $e$, then either $\psi$ fixes $e$ pointwise and must be the reflection in the plane through $e$ and the center of $P$, or $\psi$ interchanges the endpoints of $e$ and must be either the reflection in the perpendicular bisector of $e$ or the half-turn about the midpoint of $e$. Thus the stabilizer of $e$ in $G^{+}(P)$ has order $1$ or $2$. Hence the orbit of $e$ under $G^{+}(P)$ consists of either all edges of $P$ or one-half of them.

We now have two possible scenarios. Either $G^{+}(P)$ acts transitively on the edges of $P$ (which occurs when the stabilizer of at least one edge is trivial), or there are precisely two edge orbits of the same size under $G^{+}(P)$ (which occurs when all edge stabilizers of $P$ are generated by the half-turn about the midpoint of the edge). Note that if $G^{+}(P)$ acts edge transitively, then it acts sharply edge transitive (since the edge stabilizer is trivial).

Suppose $e$ is any edge of $P$ with vertices $u$ and $v$. Consider the corresponding `directed' edge, denoted by $\{u,v\}$, obtained by equipping $e$ with a `direction' pointing from $u$ to $v$. There are two directed edges associated with $e$, namely $\{u,v\}$ and $\{v,u\}$. The stabilizer of $e$ in $G^{+}(P)$ is nontrivial if and only if these two directed edges, $\{u,v\}$ and $\{v,u\}$, are equivalent under $G^{+}(P)$. In this case the two directed edges associated with any edge of $P$ are equivalent under $G^{+}(P)$, since the corresponding edge stabilizer in $G^{+}(P)$ is nontrivial as well. In other words, if the edge stabilizer of one, and hence of each, edge is trivial, then no edge of $P$ can be `inverted' modulo $G^{+}(P)$. In this case, if we begin with a directed edge and take its (directed) images under $G^{+}(P)$, we obtain a directed copy of the edge graph of $P$.

Two directed edges $\{u,v\}$ and $\{u',v'\}$ of $P$ are said to have the \emph{same orientation} if they are equivalent under $G^{+}(P)$; that is, if there exists an element $\phi \in G^{+}(P)$ such that $\phi(u) = u'$ and $\phi(v) = v'$. Otherwise the two directed edges are said to have \textit{opposite orientation}. There are always two transitivity classes of directed edges under $G^{+}(P)$. If $G^{+}(P)$ is transitive on the (undirected) edges of $P$, then, for each edge of $P$, the two corresponding directed edges lie in different transitivity classes under $G^{+}(P)$; that is, each transitivity class of directed edges contains exactly one of the two directed edges associated with any edge of $P$. In this case the two transitivity classes are the orbits, under $G^{+}(P)$, of the pair $\{u,v\}$ and $\{v,u\}$ of directed edges associated with any given edge of $P$. On the other hand, if $G^{+}(P)$ is not transitive on the (undirected) edges of $P$, then each of the two transitivity classes of (undirected) edges of $P$ gives rise to a single transitivity class of directed edges of $P$ of twice the size. In this case the two directed edges associated with an edge of $P$ always lie in the same transitivity class of directed edges.

If for one (and hence every) edge of $P$ with vertices $u$ and $v$, the corresponding directed edges $\{u,v\}$ and $\{v,u\}$ have opposite orientation (i.e. are not equivalent under $G^{+}(P)$), then we say $P$ has \emph{directed type}. If for one (and hence every) edge of $P$ with vertices $u$ and $v$, the corresponding directed edges $\{u,v\}$ and $\{v,u\}$ have the same orientation (are equivalent under $G^{+}(P)$), then $P$ has two edge transitivity classes under $G^{+}(P)$, and we say $P$ has \emph{bicolor type}.

P has directed type if $S$ is a dodecahedron and the edge length is $2$; or if $S$ is a cuboctahedron and the edge length is $1$; or if $S$ is an icosidodecahedron and the edge length is $1, 3,$ or $d$. $P$ has bicolor type if $S$ is a dodecahedron and the edge length is $3$; or if $S$ is a cuboctahedron and the edge length is $2$; or if $S$ is an icosidodecahedron and the edge length is $2, 4,$ or $2d$. Clearly, if $P$ has a Petrie dual polyhedron, it will have the same vertices and edge lengths as $P$, and so will have the same type.

It is clear from the definition of edge orientation that every regular polyhedron of index $2$ with vertices on one orbit and all edges having the same length is of either directed or bicolor type.

The relative orientation of consecutive edges of a face boundary of $P$ is determined by the change of direction between the edges, as specified in the following lemma.

\begin{lemma}
\label{lem10} When $S$ is a cuboctahedron or an icosidodecahedron, then if two consecutive edges on a face boundary of $P$ have a change of direction represented by $r$ or $l$, the two corresponding directed edges have the same orientation if $P$ has directed type, and opposite orientation if $P$ has bicolor type. If the change of direction between the edges is represented by $f$, the two corresponding directed edges have opposite orientation if $P$ has directed type, and have the same orientation if $P$ has bicolor type.

When $S$ is a dodecahedron and the edge length of $P$ is $2$ or $3$, then if two consecutive edges on a face boundary of $P$ have a change of direction represented by $hr$, $f$, or $hl$, the two corresponding directed edges have the same orientation if $P$ has directed type, and opposite orientation if $P$ has bicolor type. If the change of direction between the edges is represented by $sr$ or $sl$, the two corresponding directed edges have opposite orientation if $P$ has directed type, and the same orientation if $P$ has bicolor type.
\end{lemma}

\begin{proof}
By Lemma~\ref{lem2}, $q_{P}$ is $6$ if $S$ is a dodecahedron, and is $4$ if $S$ is a cuboctahedron or an icosidodecahedron. $P$ has $q = q_{P}$ directed edges from any given vertex $v$, namely $\{v,u_{1}\}$, $\{v,u_{2}\}$, \ldots , $\{v,u_{q}\}$, listed in cyclic (counterclockwise) order about $v$ on $S$. These directed edges are all equivalent under $G(P)$, but only alternate ones are equivalent under $G^{+}(P)$ and so have the same orientation. Thus $\{v,u_{i}\}$ and $\{v,u_{j}\}$ have the same orientation if and only if $i \equiv j (mod 2)$. If $\{u_{i},v\}$ and $\{v,u_{j}\}$ are two consecutive edges of a face boundary of $P$, then a change of direction represented by $r$ or $l$, if $S$ is a cuboctahedron or an icosidodecahedron, or by $hr$, $f$, or $hl$, if $S$ is a dodecahedron, translates into the condition $i \equiv j+1 (mod 2)$, whereas a change of direction represented by $f$, if $S$ is a cuboctahedron or an icosidodecahedron, or by $sr$ or $sl$, if $S$ is a dodecahedron, gives $i \equiv j (mod 2)$. The lemma follows when we recall that, by definition, $\{v,u_{i}\}$ and $\{u_{i},v\}$ have the same orientation (i.e., are equivalent under $G^{+}(P)$) if and only if $P$ is of bicolor type.
\end{proof}

\section{Directed Type}
\label{dirty}

We first find all polyhedra of index $2$ with directed type. As noted above, if $P$ has directed type then either $S$ is a dodecahedron and the edge length is $2$; or $S$ is a cuboctahedron and the edge length is $1$; or $S$ is an icosidodecahedron and the edge length is $1$, $3$, or $d$.

If $P$ has directed type, then $G_{F}^{+}(P)$ does not contain automorphisms of the form $\rho_{e}(F)$, as $\rho_{e}(F)$ changes the direction of the edge that is flipped, while preserving the orientation. Similarly, $\rho_{v}(F)$ results in opposite orientation for the edges of $F$ that contain $v$, while $\sigma_{1}(F)$ and $\sigma_{1}^{2}(F)$ preserve the orientation of rotated edges. From these considerations, and by referring back to Lemma~\ref{lem4}, we can tabulate all possibilities for $F$, together with allowable orientation of its edges, where $x$ and $y$ represent opposite orientations. Since $\sigma_{1}^{4}(F) \in G_{F}^{+}(P)$, it is sufficient to record the orientation pattern of faces by a string containing four entries.

\begin{table}[htb]
\centering
\begin{tabular}{|c|c|c|c|} \hline
If $G_{F}^{+}(P)$ contains &then $F$ has &and orientation pattern & \# face orbits \\
automorphisms of &shape\ldots &(for the directed &under $G^{+}(P)$ \\
the form\ldots	&&boundary of $F$)	&\\[.05in]
\hline
\hline
$\sigma_{1}(F)$, &$[a,a,a,a]$	&$.x.x.x.x.$ &2 \\
but not $\rho_{v}(F)$	&&&\\
\hline
$\sigma_{1}^{2}(F)$ and $\rho_{v}(F)$, &$[f,f,f,f]$	&$.x.y.x.y.$ &2 \\
but not $\sigma_{1}(F)$	&&&\\
\hline
$\sigma_{1}^{2}(F)$, &$[a,b,a,b]$	&$.x.x.x.x.$ or $.x.y.x.y.$	&1 \\
but not $\sigma_{1}(F)$ or $\rho_{v}(F)$ &&&\\
\hline
$\rho_{v}(F)$, &$[a,f,a',f]$ &$.x.x.y.y.$ or $.x.y.y.x.$ &1 \\
but not $\sigma_{1}^{2}(F)$ &or $[f,a,f,a']$ &or $.x.y.x.y.$ &\\
\hline
\end{tabular}
\caption{All possible face shapes when $P$ has directed type.}
\label{tab:DirectedType}
\end{table}

\begin{theorem}
\label{thm2}
There are exactly two finite regular polyhedra of index $2$ with directed type whose vertices coincide with those of a regular dodecahedron. These have edges of length $2$. One of these has type $\{5, 6\}$, with genus = $9$, and is orientable with planar faces in two orbits under $G^{+}(S)$, being pentagrams of shape $[hl,hl,hl,hl]$ and convex pentagons of shape $[f,f,f,f]$. Its Petrie dual has type $\{4, 6\}$, with genus = $12$, and is non-orientable with non-planar faces in one orbit under $G^{+}(S)$, and of shape $[hl,f,hl,f]$.
\end{theorem}

\begin{proof}
Let $P$ have directed type, with a face $F$, and $S$ be a dodecahedron. Then the edge length of $P$ is $2$, and by Lemma~\ref{lem4}, $G_{F}^{+}(P)$ does not contain any automorphism of the form $\rho_{v}(F)$. By Lemma~\ref{lem2}, $q_{P} = 6$, so there are five possible directions ($hr$, $sr$, $f$, $sl$, and $hl$) in which a face boundary of $P$ may continue at any vertex. Note that for a dodecahedron, unlike all other allowable structures for $S$, a direction of $f$ does not imply that the two connecting edges lie in a common plane through the origin.

We first suppose that $P$ has two face orbits under $G^{+}(P)$. From Table \ref{tab:DirectedType}, we see that this requires that $\sigma_{1}(F) \in G_{F}^{+}(P)$, so that all faces of $P$ have shape $[a,a,a,a]$, where $a$ must have a different value for each of the two face orbits, and that each directed edge in the face boundary has the same orientation. By Lemma~\ref{lem10}, the value of $a$ must be $hr$, $f$, or $hl$. We therefore determine the number of vertices, $p_{F}$, a face $F$ with shape $[hr,hr,hr,hr]$, $[f,f,f,f]$, or $[hl,hl,hl,hl]$ would have. Since $p_{F} = p_{P}$ for all faces of $P$, we require that two different face shapes have the same value of $p_{F}$. In order to determine these face shapes, we need to specify the orientation of the starting edge. We say that an edge $\{u,v\}$ of $P$ has left orientation if the corresponding edge path $\{u,w,v\}$ of length $2$ on $S$ bends left at $w$, and has right orientation if the path bends to the right at $w$, and we adopt the convention that the starting edge has left orientation.

Then a face $F$ of shape $[hr,hr,hr,hr]$ is a vertex-figure of $S$, with $p_{F} = 3$; a face of shape $[hl,hl,hl,hl]$ is a pentagram on a face of $S$, with $p_{F} = 5$; and a face of shape $[f,f,f,f]$ is a convex pentagon taking alternate vertices of a Petrie polygon of $S$, also with $p_{F} = 5$. Thus if $P$ has two face orbits under $G^{+}(P)$, it must have shape $[hl,hl,hl,hl] \& [f,f,f,f]$.

Suppose now that $P$ has one face orbit, so that $G_{F}^{+}(P)$ contains $\sigma_{1}^{2}(F)$ but not $\sigma_{1}(F)$. By Table \ref{tab:DirectedType}, $F$ has shape $[a,b,a,b]$ (allowing $a = b$), where the directed boundary of $F$ either has no change in orientation, or changes orientation at every vertex. However, the latter case can not occur, for if the boundary of $F$ changed orientation at every vertex, then the boundary would be confined to the edges of a cube inscribed in the dodecahedron that forms part of the regular compound of five cubes whose vertices coincide with the vertices of a regular dodecahedron. Any two of the cubes of this compound have only two (antipodal) vertices in common, and no edge of $F$ joins those two vertices. So if the boundary of $F$ changed orientation at every vertex, $P$ might be a compound but could not be a polyhedron.

Thus the directed boundary of $F$ has no change in orientation, and has shape $[a,b,a,b]$, where now $a$ and $b$ must differ since $\sigma_{1}(F) \notin G_{F}^{+}(P)$. By Lemma~\ref{lem10}, $a, b \in \{hr, f, hl\}$. Of the six possible face shapes, we have that $[hl,hr,hl,hr]$, $[f,hr,f,hr]$, and $[f,hl,f,hl]$ are the same faces as $[hr,hl,hr,hl]$, $[hr,f,hr,f]$, and $[hl,f,hl,f]$, with the respective starting edges being successive edges in the boundary of $F$.

By Lemma~\ref{lem6}, a polyhedron of shape $[hr,hl,hr,hl]$ has Petrie polygons of shape $[hr,hr,hr,hr]$ or $[hl,hl,hl,hl]$. We have just verified that these Petrie polygons do not revisit any vertex, so by Lemma 3.2 of \cite{cs}, the polyhedron of shape $[hr,hl,hr,hl]$ has a polyhedral Petrie-dual. By Lemma~\ref{lem7}, a polyhedron of shape $[hr,f,hr,f]$ or $[hl,f,hl,f]$ has a polyhedral Petrie-dual with two face orbits. The only candidate for any of these Petrie-duals is the structure of shape $[hl,hl,hl,hl] \& [f,f,f,f]$, which is the Petrie-dual of the structure of shape $[hl,f,hl,f]$.

It is the case that these two structures are both regular polyhedra of index $2$, and the verification of their combinatorial regularity is given in Appendix 2 of \cite{cut}. A representative face from each face orbit is shown in the left column of Figure \ref{fig3}.
\end{proof}

\begin{theorem}
\label{thm3}
There are no finite regular polyhedra of index $2$ with directed type whose vertices coincide with those of a regular cuboctahedron. There are exactly two finite regular polyhedra of index $2$ with directed type whose vertices coincide with those of a regular icosidodecahedron. These have edge length $d = (1+\sqrt{5})/2$. One has type $\{5, 4\}$, with genus = $4$, and is orientable with planar faces in two orbits under $G^{+}(S)$, the faces being convex pentagons of shape $[r,r,r,r]$ and pentagrams of shape $[l,l,l,l]$. The other is its Petrie dual with type $\{6, 4\}$, and genus = $12$, and is non-orientable with non-planar hexagonal faces in one orbit under $G^{+}(S)$, and is of shape $[r,l,r,l]$.
\end{theorem}

\begin{proof}
Let $P$ have directed type, and $S$ be a cuboctahedron or an icosidodecahedron. By Lemma~\ref{lem2}, $q_{P} = 4$ so there are three possible directions $(r, f,$ and $l)$ in which a face boundary of a polyhedron may continue at any vertex.

Suppose first that $P$ has two face orbits under $G^{+}(P)$. From Table \ref{tab:DirectedType} we see then that all faces of $P$ have shape $[a,a,a,a]$, where $a$ must have a different value for each of the two face orbits. Note here that any face of shape $[r,r,r,r]$ or $[l,l,l,l]$ has all its directed boundary edges in the same orientation by Lemma~\ref{lem10}, and so can not be traversed in the opposite direction, as starting edges must have the same orientation; thus $[r,r,r,r]$ and $[l,l,l,l]$ represent different faces. We first determine the number of vertices, $p_{F}$, a face $F$ with shape $[r,r,r,r]$, $[f,f,f,f]$, or $[l,l,l,l]$ would need to have for a given edge length when $S$ is a cuboctahedron or an icosidodecahedron. All faces of $P$ must have the same value of $p_{F}$. 

For edge length $1$, any face of shape $[r,r,r,r]$ or $[l,l,l,l]$ would be a face of $S$ (we assume that the starting edge is such that $[r,r,r,r]$ represents a triangular face), and any face of shape $[f,f,f,f]$ would be a circumference of $S$, parallel to a triangular face when $S$ is a cuboctahedron, and parallel to a pentagon face when $S$ is an icosidodecahedron. Thus $[r,r,r,r]$, $[l,l,l,l]$, and $[f,f,f,f]$ each have a different value of $p_{F}$, namely $3, 4, 6,$ respectively, for the cuboctahedron, and $3, 5, 10,$ respectively, for the icosidodecahedron.

If $S$ is an icosidodecahedron and the edge length is $3$, then for appropriate choice of starting edge the faces with shape $[r,r,r,r]$ are pentagrams centered on pentagon faces of $S$, the faces with shape $[l,l,l,l]$ are triangles centered on triangle faces of $S$, and the faces with shape $[f,f,f,f]$ go three times round a circumference of $S$, parallel to a pentagon face of $S$, giving a star-decagon $\{\frac{10}{3}\}$ inscribed in the circumference. Thus the three face shapes again have different values of $p_{F}$, being $5, 3,$ and $10$, respectively.

Finally, if $S$ is an icosidodecahedron and the edge length is $d$, then the faces with shape $[r,r,r,r]$ and $[l,l,l,l]$ are pentagrams and pentagons, each centered on a pentagon face of $S$, and the faces with shape $[f,f,f,f]$ are planar hexagons contained in a circumference of $S$ that is parallel to a triangular face of $S$ (but does not run along edges of $S$). The face shapes $[r,r,r,r]$ and $[l,l,l,l]$ each have $p_{F} = 5$, and the face shape $[f,f,f,f]$ has $p_{F} = 6$.

It follows that if $S$ is a cuboctahedron or an icosidodecahedron, and $P$ has directed type with two face orbits under $G^{+}(P)$, then the only candidate structure for $P$ occurs when $S$ is an icosidodecahedron and $P$ has shape $[r,r,r,r] \& [l,l,l,l]$ with edge length of $d$.

Suppose now that $P$ has one face orbit under $G^{+}(P)$. By Lemma~\ref{lem5}, and the above, no Petrie polygon of $P$ has shape $[f,f,f,f]$, so, by Lemma~\ref{lem7}, no face of $P$ has shape $[a,f,b,f]$ or $[f,c,f,d]$. We then have, from Table \ref{tab:DirectedType}, that $G_{F}^{+}(P)$ contains $\sigma_{1}^{2}(F)$, but not $\sigma_{1}(F)$, for all $F$. Further, $P$ can not have a face of shape $[a,a,a,a]$, as all such faces were examined above, and in each case $G_{F}^{+}(P)$ contained $\sigma_{1}(F)$. Thus the only possibilities for the shape of $P$ are $[r,l,r,l]$ or $[l,r,l,r]$. These represent the same face, by Lemma~\ref{lem10}, with the respective starting edges being successive edges in the face boundary. Thus $[r,l,r,l]$ is the only possible face shape for which $G_{F}^{+}(P)$ contains $\sigma_{1}^{2}(F)$ but not $\sigma_{1}(F)$. We now find allowable edge lengths of $P$ by showing that if $P$ has shape $[r,l,r,l]$, for any given edge length, then the Petrie-dual, $Q$, of $P$ is also a regular polyhedron of index $2$. By Lemma 3.2 of \cite{cs}, it is sufficient to show that a Petrie polygon of $P$ does not revisit a vertex. From Lemma~\ref{lem6}, every Petrie polygon of $P$ has shape $[r,r,r,r]$ or $[l,l,l,l]$, and we have just checked all such polygons for all allowable edge lengths and found that none of them revisit a vertex. Thus $Q$ is a regular polyhedron of index $2$ with directed type, and by Lemma~\ref{lem6} it has shape $[r,r,r,r]$, $[l,l,l,l]$, or $[r,r,r,r] \& [l,l,l,l]$. We have already verified that there is only one possible such polyhedron (obtained when $S$ is an icosidodecahedron, and the edge length is $d$), and so the only candidate with one face orbit under $G^{+}(P)$ is its Petrie-dual.

It is the case that both these structures are regular polyhedra of index $2$, and the verification of their combinatorial regularity is given in Appendix 2 of \cite{cut}. Their representative faces are shown in the left column of Figure \ref{fig4}.
\end{proof}

\section{Bicolor Type}
\label{bicty}

We turn now to structures with bicolor type, which are the only remaining potential polyhedra. As noted earlier, $P$ has bicolor type if (and only if) $S$ is a dodecahedron and the edge length is $3$; or $S$ is a cuboctahedron and the edge length is $2$; or $S$ is an icosidodecahedron and the edge length is $2, 4,$ or $2d$.

We first eliminate the possibility that $G_{F}^{+}(P)$ contains automorphisms of the form $\rho_{e}(F)$ if $P$ has one face orbit under $G^{+}(P)$. We begin with a general lemma.

For an edge $e$ of $P$ we let $G_{e}(P)$ and ${G_{e}}^{+}(P)$ denote the stabilizer of $e$ in $G(P)$ and $G^{+}(P)$, respectively. Clearly, ${G_{e}}^{+}(P)$ is a subgroup of $G_{e}(P)$.

\begin{lemma}
\label{lem11} Let $P$ have all its vertices in one orbit under $G(P)$ and have all its edges of the same length. Then $G(P)$ is edge transitive and the stabilizer $G_{e}(P)$ of an edge $e$ has order $2$. In particular, if $P$ is of bicolor type then $G_{e}(P)$ is generated by a half-turn and equals ${G_{e}}^{+}(P)$; if $P$ is of directed type then $G_{e}(P)$ is generated by a plane reflection and ${G_{e}}^{+}(P)$ is trivial.
\end{lemma}

\begin{proof}
As pointed out earlier, our assumptions on $P$ imply that $G(P)$ is edge transitive. Since $P$ has exactly $|G(P)|/2$ edges, the stabilizer in $G(P)$ of an edge $e$ must have order $2$. By definition, if $P$ is of bicolor type, then for every edge $e$, ${G_{e}}^{+}(P)$ is generated by a half-turn about the midpoint of $e$. Since $G_{e}(P)$ and ${G_{e}}^{+}(P)$ have the same order, they must coincide. On the other hand, if $P$ is of directed type, then ${G_{e}}^{+}(P)$ is trivial. Since the generating element of $G_{e}(P)$ is an involution, it must be either a plane reflection, a half-turn, or the central inversion. The last is impossible, since edges of $P$ can not pass through the center. Neither can it be a half-turn, since then it would lie in ${G_{e}}^{+}(P)$. Thus $G_{e}(P)$ is generated by the reflection in the perpendicular bisector of $e$.
\end{proof}

\begin{lemma}
\label{lem12} If $P$ has bicolor type, and $P$ has one face orbit under $G^{+}(P)$, then for any face, $F$, of $P$, $G_{F}^{+}(P)$ can not contain an automorphism of the form $\rho_{e}(F)$.
\end{lemma}

\begin{proof}
By Lemmas 5.1 and 5.3 of \cite{cs}, if $P$ has one face orbit under $G_{F}^{+}(P)$ and hence under $G(P)$, then $|G_{F}(P)| = 2|G_{F}^{+}(P)| = p$, so necessarily $\sigma_{1}^{2}(F) \in G_{F}(P)$. Now suppose $G_{F}^{+}(P)$ contains an automorphism of the form $\rho_{e}(F)$. Then by Lemma~\ref{lem4}, $\sigma_{1}^{2}(F) \notin G_{F}^{+}(P)$. However, $\sigma_{1}^{2}(F) = \rho_{e}(F)\rho_{e'}(F)$, where $e'$ is an edge of $F$ adjacent to $e$, and $\rho_{e'}(F) \in G_{F}(P)$ since both $\sigma_{1}^{2}(F)$ and $\rho_{e}(F) \in G_{F}(P)$, but $\rho_{e'}(F) \notin G_{F}^{+}(P)$ since $\rho_{e}(F)$ lies in $G_{F}^{+}(P)$ but $\sigma_{1}^{2}(F)$ does not. This implies that $\rho_{e'}(F)$ is not a half-turn; it must therefore be a plane reflection. This contradicts Lemma~\ref{lem11}, since all stabilizers of edges in $G(P)$ must be generated by half-turns if $P$ is of bicolor type.
\end{proof}

We next eliminate the possibility that $G_{F}^{+}(P)$ contains $\sigma_{1}(F)$, or automorphisms of the form $\rho_{v}(F)$.

\begin{lemma}
\label{lem13} If $P$ has bicolor type, and $F$ is a face of $P$, then $G_{F}^{+}(P)$ contains $\sigma_{1}^{2}(F)$, but does not contain $\sigma_{1}(F)$, or automorphisms of the form $\rho_{v}(F)$. Moreover, adjacent edges in the boundary of $F$ have opposite orientation.
\end{lemma}

\begin{proof}
By definition of bicolor type, $P$ has two edge transitivity classes under $G^{+}(P)$, of equal size. The orientation of each edge specifies which transitivity class it belongs to. If the boundary of $F$ (and thus also of every face in the same orbit as $F$) consists only of edges of the same orientation, it would be necessary for $P$ to have two face orbits under $G^{+}(P)$, and for the boundary of every face of $P$ not in the same orbit as $F$ to consist only of edges having the opposite orientation to the edges in the boundary of $F$. This would result in two face orbits with no edges in common, which is not possible as the underlying map of $P$ is flag transitive. So $F$ must have edges of opposite orientation in its boundary. In particular, $G_{F}^{+}(P)$ does not contain $\sigma_{1}(F)$.

Suppose now that $G_{F}^{+}(P)$ contains automorphisms of the form $\rho_{v}(F)$. By Lemma~\ref{lem4}, $S$ must be a cuboctahedron or an icosidodecahedron. In that case, $P$ can not have any face of shape $[f,f,f,f]$, as its boundary would consist only of edges in the same transitivity class by Lemma~\ref{lem10}. So, by Lemma \ref{lem4}, $P$ has one face orbit under $G^{+}(P)$. Thus Lemma \ref{lem7} applies, giving that the Petrie dual, $Q$, of $P$ is also a regular polyhedron of index $2$, with a face, $F'$, of shape $[f,f,f,f]$. $Q$ has the same vertices and edge lengths as $P$, so all the boundary edges of $F'$ have the same orientation, by Lemma~\ref{lem10}. We have just seen that this is not possible, so it follows that $G_{F}^{+}(P)$ does not contain automorphisms of the form $\rho_{v}(F)$. Therefore, by Lemmas \ref{lem12} and \ref{lem4}, $G_{F}^{+}(P)$ contains $\sigma_{1}^{2}(F)$, and since not all edges in the boundary of $F$ have the same orientation, it must be that adjacent edges have opposite orientation.
\end{proof}

Since, by Lemma~\ref{lem13} and Theorems~\ref{thm1}, \ref{thm2}, and \ref{thm3}, $G_{F}^{+}(P)$ contains $\sigma_{1}^{2}(F)$ for all faces $F$ of $P$, we may describe the faces of $P$ by a $2$-symbol shape notation, $[a,b]$. As in \cite{cs}, we use $[a,b]$ to indicate the shape $[a,b,a,b]$. Consequently, if $P$ has bicolor type, then from Lemmas \ref{lem4}, \ref{lem12}, and \ref{lem13} we can set out all possible shapes and edge orientation of faces of $P$. As before, $x$ and $y$ represent opposite orientations. Note that if $P$ has one face orbit under $G^{+}(P)$ and $F$ has shape $[a,a']$, then for any edge, $e$, of $F$, the half-turn $\psi \in {G_{e}}^{+}(P)$ maps $F$ to itself, giving $\psi = \rho_{e}(F)$, which contradicts Lemma~\ref{lem12}.

\begin{table}[htb]
\centering
\begin{tabular}{|c|c|c|c|} \hline
If $G_{F}^{+}(P)$ contains &then $F$ has 	&and boundary & \# face orbits \\
automorphisms of the form\ldots &shape\ldots &orientation pattern &under $G^{+}(P)$\\[.05in]
\hline
\hline
$\sigma_{1}^{2}(F)$ and $\rho_{e}(F)$, &$[a,a']$	&$.x.y.$	&2 \\
but not $\sigma_{1}(F)$ or $\rho_{v}(F)$ &&&\\
\hline
$\sigma_{1}^{2}(F)$, &$[a,b]$	&$.x.y.$ &1 \\
but not $\sigma_{1}(F)$ or $\rho_{e}(F)$ or $\rho_{v}(F)$	&($b\neq a'$) &&\\
\hline
\end{tabular}
\caption{All possible face shapes when $P$ has bicolor type.}
\label{tab:BicolorType}
\end{table}

\begin{theorem}
\label{thm4}
There are exactly two finite regular polyhedra of index $2$ with bicolor type whose vertices coincide with those of a regular dodecahedron. These have edges of length $3$. One of these has type $\{4, 6\}$, with genus = $6$, and is orientable with non-planar faces in one orbit under $G^{+}(S)$, and has shape $[hl,f]$. The other is its Petrie dual with type $\{10, 6\}$, and genus = $30$, and is non-orientable with two kinds of antiprismatic decagons as faces, in two face orbits under $G^{+}(S)$, and has shape $[hl,hr] \& [f,f]$.
\end{theorem}

\begin{proof}
Let $P$ have bicolor type, and $S$ be a dodecahedron. Thus the edge length of $P$ is $3$. By Lemma~\ref{lem2}, $q_{P} = 6$, so there are five possible directions $(hr, sr, f, sl,$ and $hl)$ in which a face boundary of $P$ may continue at any vertex. Let $F$ be any face of $P$. From Table \ref{tab:BicolorType}, $F$ has shape $[a,b]$, where the boundary of $F$ changes orientation at every vertex, so $a, b \in \{hl, f, hr\}$, by Lemma~\ref{lem10}.

We first find $p_{F}$ for the cases when $P$ has two face orbits under $G^{+}(P)$, so that $F$ has shape $[a,a']$. In order to make the face shape notation precise, we identify a starting edge, by specifying the orientation on $S$. When $S$ is a dodecahedron and the edges of $P$ have length $3$, we say that an edge $\{u,v\}$ of $P$ has left orientation if the edge path $\{u,w,t,v\}$ of $S$ of minimum length connecting $u$ and $v$ bends left at $w$, and has right orientation if the path on $S$ bends to the right at $w$. If the starting edge of $F$ has left orientation, then $[hr,hl]$ represents an antiprismatic hexagonal polygon with its vertices at the vertices of the vertex-figures of a pair of antipodal vertices of $S$, and has $p_{F} = 6$; $[hl,hr]$ represents an antiprismatic decagon with its vertices at the vertices of antipodal faces of $S$, and has $p_{F} = 10$; and $[f,f]$ represents a star-polygon $\{\frac{10}{3}\}$ (also an antiprismatic decagon) inscribed in a Petrie polygon of $S$, and has $p_{F} = 10$.

Thus the only candidate for a polyhedron with two face orbits (and with bicolor type and where $S$ is a dodecahedron) is the structure with shape $[hl,hr] \& [f,f]$.

Referring to Table \ref{tab:BicolorType}, when $P$ has one face orbit under $G^{+}(P)$, $F$ has shape $[a,b]$, where $b \neq a'$. Here $[hl,hl]$, $[f,hr]$, and $[f,hl]$ are the same faces as $[hr,hr]$, $[hl,f]$, and $[hr,f]$ traversed in the opposite direction, so $P$ has three possible shapes. We show that if any of these candidate structures for $P$ are polyhedral, then the Petrie dual, $Q$, is also polyhedral and has two face orbits. This is so for the structures of shape $[hl,f]$ and $[hr,f]$ by Lemma~\ref{lem7}. If $P$ has shape $[hr,hr]$, then by Lemma~\ref{lem6} the Petrie polygons of $P$ have shape $[hr,hl]$ or $[hl,hr]$. By Lemma 3.2 of \cite{cs}, $Q$ is a polyhedron if any of the Petrie polygons of $P$ do not revisit vertices of $P$, and we have just ascertained that this is the case for both $[hr,hl]$ and $[hl,hr]$.

So the only possible index $2$ polyhedra of bicolor type when $S$ is a dodecahedron are $[hl,hr] \& [f,f]$, and its Petrie dual, $[hl,f]$. It is the case that both these structures are polyhedra of index $2$, and the verification of their combinatorial regularity is given in Appendix 2 of \cite{cut}. Diagrams of their representative faces are in the right column of Figure \ref{fig3}.
\end{proof}

\begin{theorem}
\label{thm5}
There are no finite regular polyhedra of index $2$ with bicolor type whose vertices coincide with those of a regular cuboctahedron. There are exactly two finite regular polyhedra of index $2$ with bicolor type whose vertices coincide with those of a regular icosidodecahedron. These have edge length $1+\sqrt{5}$ (= $2d$). One has type $\{6, 4\}$, with genus = $6$, and is orientable with non-planar hexagonal faces in one orbit under $G^{+}(S)$, and has shape $[r,r]$. The other is its Petrie dual with type $\{10, 4\}$, and genus = $20$, and is non-orientable with two kinds of antiprismatic decagons as faces, in two face orbits under $G^{+}(S)$, and has shape $[r,l] \& [l,r]$.
\end{theorem}

\begin{proof}
Let $P$ have bicolor type, and $S$ be a cuboctahedron or an icosidodecahedron. Then $q_{P} = 4$, and there are three possible directions $(r, f,$ and $l)$ in which a face boundary of $P$ may continue at any vertex. Let $F$ be a face of $P$.

Suppose first that $P$ has two face orbits under $G^{+}(P)$. By Lemmas \ref{lem10} and \ref{lem13}, no face of $P$ has shape $[f,f]$. Thus, from Table \ref{tab:BicolorType}, we have that the only possible shapes for faces of $P$ are $[r,l]$, and $[l,r]$. We now calculate $p_{F}$ for these face shapes. We require that the faces in each orbit have the same value of $p_{F}$.

We specify a starting edge for the face shape by distinguishing between the two forms of orientation on $P$. We assume that the starting edge of $F$ is a directed edge $\{u,v\}$ of $P$ such that in the edge or face-diagonal path $\{u,w,\ldots,v\}$ of $S$ of minimum length connecting $u$ and $v$, the triangle face of $S$ adjoining $\{u,w\}$ (or the centroid of the pentagon face of $S$ traversed by $\{u,w\}$, if $P$ has edge length $2d$) is to the left of $\{u,w\}$. With this convention, $P$ is fully determined by its face shape, vertex position, and edge length.

If $S$ is a cuboctahedron and $P$ has edge length $2$, then a face with shape $[r,l]$ is an antiprismatic hexagon with vertices in two opposite triangular faces of $S$, and has $p_{F} = 6$; and a face with shape $[l,r]$ is a nonplanar square with its vertices at opposite square faces of $S$ and its diagonals given by diagonals of those square faces, and has $p_{F} = 4$.

If $S$ is an icosidodecahedron and $P$ has edge length $2$, then a face with shape $[r,l]$ is an antiprismatic decagon positioned along a decagon circumference of $S$, and has $p_{F} = 10$; and a face with shape $[l,r]$ is an antiprismatic hexagon positioned along a circumference of $S$ parallel to a triangle face of $S$, and has $p_{F} = 6$.

If $S$ is an icosidodecahedron and $P$ has edge length $4$, then a face with shape $[l,r]$ is an antiprismatic decagon (over a pentagram) positioned along a decagonal circumference of $S$ and with its vertices in the two pentagonal faces of $S$ parallel to the circumference, and has $p_{F} = 10$; and a face with shape $[r,l]$ is an antiprismatic hexagon with its vertices in two opposite triangular faces of $S$, and has $p_{F} = 6$.

Finally, when $S$ is an icosidodecahedron and $P$ has edge length $2d$, then a face with shape $[r,l]$ is an antiprismatic decagon positioned along a circumference of $S$ parallel to a pentagon face and with its vertices at opposite pentagon faces of $S$, and has $p_{F} = 10$; and a face with shape $[l,r]$ is an antiprismatic decagon (over a pentagram) positioned along a circumference of $S$ parallel to a pentagon face, and has $p_{F} = 10$.

It follows that if $S$ is a cuboctahedron or an icosidodecahedron, and $P$ has bicolor type with two face orbits under $G^{+}(P)$, then the only candidate for $P$ is the structure with shape $[r,l] \& [l,r]$ and edge length $2d$, with $S$ an icosidodecahedron.

If $P$ has one face orbit under $G^{+}(P)$, then, from Table \ref{tab:BicolorType}, $F$ has shape $[a,b]$, where $b \neq a'$. Further, $F$ does not have shape $[a,f]$ or $[f,a]$, by Lemma~\ref{lem7}. Thus the shape of $F$ can only be $[r,r]$ or $[l,l]$, which are the same shape traversed in opposite directions. So $F$ must have shape $[r,r]$. We now find allowable edge lengths of $P$ by showing that if $P$ has shape $[r,r]$, for any given edge length, then the Petrie-dual, $Q$, of $P$ is also a regular polyhedron of index $2$. By Lemma 3.2 of \cite{cs}, it is sufficient to show that a Petrie polygon of $P$ does not revisit a vertex. From Lemma~\ref{lem6} any Petrie polygon of $P$ has shape $[r,l]$ or $[l,r]$, and we have just checked all such polygons for all allowable edge lengths and found that none of them revisit a vertex. Thus $Q$ is a regular polyhedron of index $2$ with bicolor type. Since $P$ is the only candidate for such a polyhedron with one face orbit under $G^{+}(P)$, $Q$ must have two face orbits, and we have already found that there is only one possible polyhedron with two face orbits.

Thus we have that the only two possible regular polyhedra of index $2$ occur when $S$ is an icosidodecahedron and $P$ has edge length $2d$, and they have shapes $[r,l] \& [l,r]$ and $[r,r]$. The latter shape gives skew hexagonal faces, with three vertices at vertices of triangle faces of $S$. It is the case that both these structures are regular polyhedra of index $2$, and the verification of their combinatorial regularity is given in Appendix 2 of \cite{cut}. Diagrams of their representative faces are in the right column of Figure \ref{fig4}.
\end{proof}

Thus there are precisely $10$ finite regular polyhedra of index $2$ with vertices on one orbit under the full symmetry group. These polyhedra all have icosahedral symmetry, and all the polyhedra have faces such that the squares of the combinatorial rotation about the face lies in the geometric rotation group (that is, $\sigma_{1}^{2}(F) \in G^{+}(P)$, for all $F$). Three of the polyhedra have planar faces.

The underlying maps of these polyhedra, and verification of their regularity, are in Appendix 2 of \cite{cut}.
 
Table \ref{tab:OneOrbit} lists these polyhedra in Petrie pairs in the order in which they occur in this paper. Recall that a polyhedron of type $\{p, q\}_{r}$ has a Petrie dual of type $\{r, q\}_{p}$. We see that each Petrie pair consists of one orientable polyhedron and one non-orientable polyhedron; one of these has one face orbit under $G^{+}(P)$, and the other has two face orbits under $G^{+}(P)$.

\begin{table}[htb]
\centering
\begin{tabular}{|c|c|c|c|c|c|c|} \hline
Type &Face Vector &Edge &Shape of $P$	&Map of &Figure	&Notes \\
$\{p, q\}_{r}$	&$(f_{0}, f_{1}, f_{2})$ &length	&& \cite{con} &&\\[.05in]
\hline
\hline
$\{6, 6\}_{6}$	&$(20, 60, 20)$	&$1, 4$	&$[r,r]$	&$R11.5$	&\ref{fig2}	&Planar faces; \\
&&&&&&self-dual map \\
\hline
$\{6, 6\}_{6}$	&$(20, 60, 20)$	&$1, 4$	&$[r,l] \& [l,r]$	&$N22.3$	&\ref{fig2}	&One face orbit \\
&&&&&&under $G(P)$ \\
\hline
$\{4,6\}_{5}$	&$(20, 60, 30)$	&$2$	&$[hl,f]$	&$N12.1$	&\ref{fig3} & \\
\hline
$\{5,6\}_{4}$	&$(20, 60, 24)$	&$2$	&$[f,f] \& [hl,hl]$	&$R9.16$	&\ref{fig3}	&Planar faces \\
\hline
$\{6,4\}_{5}$	&$(30, 60, 20)$	&$d$	&$[r,l]$	&$N12.1^*$	&\ref{fig4} & \\
\hline
$\{5,4\}_{6}$	&$(30, 60, 24)$	&$d$	&$[r,r] \& [l,l]$	&$R4.2^*$	&\ref{fig4}	&Planar faces \\
\hline
$\{4,6\}_{10}$	&$(20, 60, 30)$	&$3$	&$[hl,f]$	&$R6.2$	&\ref{fig3} & \\
\hline
$\{10,6\}_{4}$	&$(20, 60, 12)$	&$3$	&$[f,f] \& [hl,hr]$	&$N30.11^*$	&\ref{fig3} & \\
\hline
$\{6,4\}_{10}$	&$(30, 60, 20)$	&$2d$	&$[r,r]$	&$R6.2^*$	&\ref{fig4} & \\
\hline
$\{10,4\}_{6}$	&$(30, 60, 12)$	&$2d$	&$[r,l] \& [l,r]$	&$N20.1^*$	&\ref{fig4} & \\
\hline
\end{tabular}
\caption{The finite regular polyhedra of index $2$ with vertices on one orbit.}
\label{tab:OneOrbit}
\end{table}

Note that in the Conder, Dobesanyi~\cite{codo} and Conder \cite{con} classification of regular maps, `$R$' indicates an orientable, non-chiral (i.e., reflexible) map and `$N$' indicates a non-orientable map. The number before the period is the genus of the map, and an asterisk indicates the dual of the map. Thus $N30.11^{*}$ is the dual of the $11^{th}$ non-orientable map of genus $30$ in the listing. 

The two polyhedra of type $\{6, 6\}_{6}$ are neither combinatorially isomorphic, nor dual to each other: one is orientable, the other non-orientable. However the orientable polyhedron of type $\{6, 6\}_{6}$ is combinatorially self-dual. Moreover, the polyhedra of types $\{6, 4\}_{5}$ and $\{4, 6\}_{5}$ are combinatorially dual, as are the polyhedra of types $\{6, 4\}_{10}$ and $\{4, 6\}_{10}$. Finally, the two orientable polyhedra with $p = 5$, of type $\{5, 4\}_{6}$ and $\{5, 6\}_{4}$ (which both have planar faces) are combinatorially dual to the orientable polyhedra, of type $\{4, 5\}_{6}$ and $\{6, 5\}_{4}$ respectively, with vertices on two orbits described in \cite{cs}.

Figures \ref{fig2}--\ref{fig4} show a representative face for each face orbit under $G(P)$ of these polyhedra. Polyhedra in the same vertical column are Petrie-duals. Faces in the same horizontal row are related by the duality described in the next section. In this manner the ten polyhedra fall naturally into three sets.

Taken in conjunction with \cite{cs}, this completes the classification of all regular finite polyhedra of index $2$.

\begin{figure} [hp]
\centering
\includegraphics[trim = 50mm 55mm 45mm 45mm, clip, scale = 0.8]{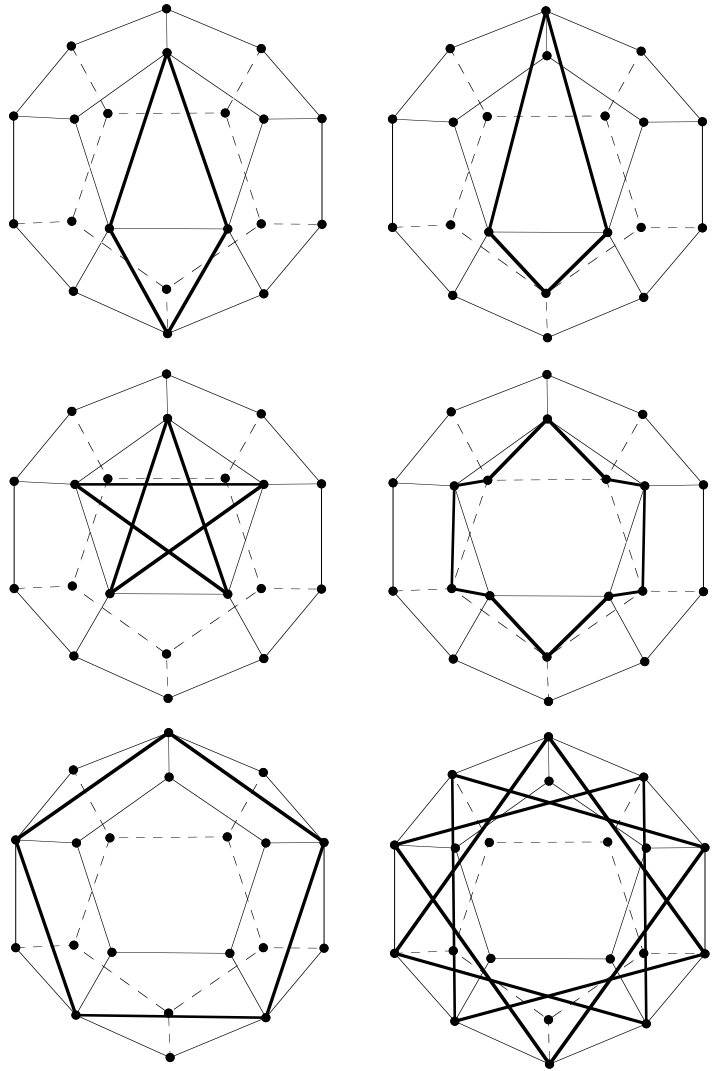}\\
\caption{The six represntative faces of the four polyhedra with edges of equal length and vertices coinciding with those of a dodecahedron. Those in the left column have directed type, and those in the right column have bicolor type. The top left has type $\{4,6\}_{5}$ and shape $[hl,f]$ and is non-orientable. Below it are the two face orbits of its Petrie-dual of type $\{5,6\}_{4}$ and shape $[hl,hl]\&[f,f]$, which is orientable. The top right has type $\{4,6\}_{10}$ and shape $[hl,f]$ and is orientable. Below it are the two face orbits of its Petrie-dual of type $\{10,6\}_{4}$ and shape $[hl,hr]\&[f,f]$, which is non-orientable.}
\label{fig3}
\end{figure}

\begin{figure} [hp]
\centering
\includegraphics[trim = 40mm 50mm 44mm 40mm, clip, scale = 0.8]{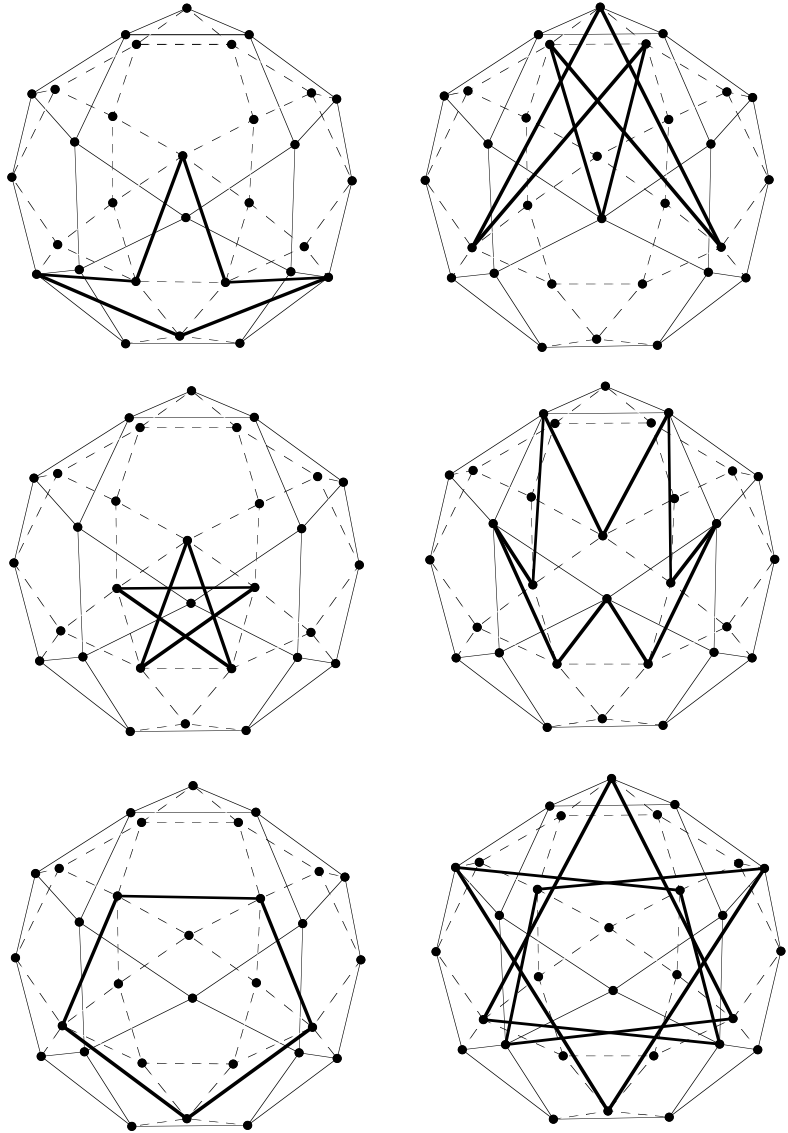}\\
\caption{The six represntative faces of the four polyhedra with edges of equal length and vertices coinciding with those of an icosidodecahedron. Those in the left column have directed type, and those in the right column have bicolor type. The top left has type $\{6,4\}_{5}$ and shape $[r,l]$ and is non-orientable. Below it are the two face orbits of its Petrie-dual of type $\{5,4\}_{6}$ and shape $[r,r]\&[l,l]$, which is orientable. The top right has type $\{6,4\}_{10}$ and shape $[r,r]$ and is orientable. Below it are the two face orbits of its Petrie-dual of type $\{10,4\}_{6}$ and shape $[r,l]\&[l,r]$, which is non-orientable.}
\label{fig4}
\end{figure}

\section{A Non-Petrie Duality}
\label{npdua}

Each pair of faces in any horizontal row of Figures \ref{fig2}--\ref{fig4} (or in any vertical column of Figures 2--7 of \cite{cs}) is related in the following manner: given a regular polyhedron $P$, with a face $F$ having boundary $\{v_{1}, v_{2}, v_{3}, v_{4}, v_{5}\ldots\}$, we can construct a corresponding polygon, which we designate $C(F)$, to have boundary $\{v_{1}, v'_{2}, v_{3}, v'_{4}, v_{5}\ldots\}$, where $v'_{i}$ is the vertex opposite $v_{i}$. If $P$ is centrally symmetric, then $v'_{i}$ is a vertex of $P$.

By choosing a different starting vertex, a different copy of $C(F)$ can be obtained that is a point reflection of the first through the center of $P$. Thus $C(F)$ is unique up to the action of $G(P)$. If $F$ has shape $[a,b,c,d]$, then $C(F)$ has shape $[a,b',c,d']$, for, as shown in Figure \ref{fig5}, the three pairs of consecutive projected edges $\{u, w, u'\}$ are half circumferences of the sphere of projection, so that the change of direction along $\{v_{1}, v'_{2}, v_{3}\}$ is opposite to that along $\{v_{1}, v_{2}, v_{3}\}$, whereas the change of direction along $\{v'_{2}, v_{3}, v'_{4}\}$ is the same as that along $\{v_{2}, v_{3}, v_{4}\}$.

\begin{figure} [ht]
\centering
\includegraphics[trim = 55mm 105mm 60mm 100mm, clip, scale = 0.6]{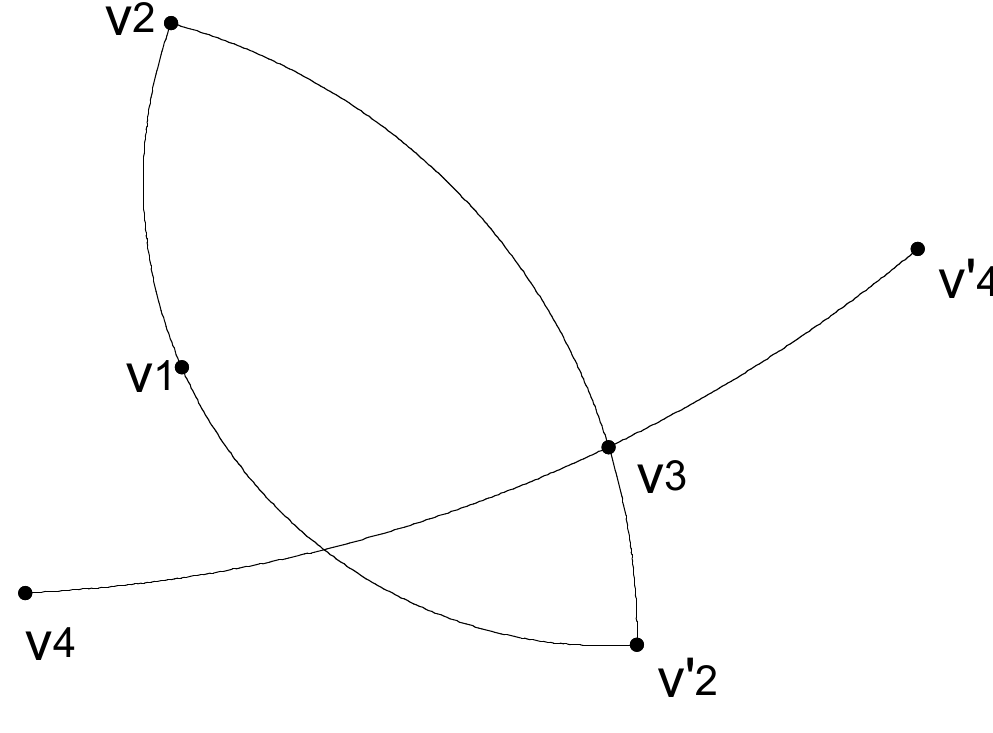}\\
\caption{}
\label{fig5}
\end{figure}

Given $P$, we form $C(P)$ similarly, with the additional requirement that if an edge, $e$, of $P$ subtends two faces $F$ and $F'$ then the boundaries of $C(F)$ and $C(F')$ will share an edge of $C(P)$ corresponding to $e$. This construction gives two possibilities for $C(P)$, but we do not distinguish between them as they are geometrically equivalent, one being a point reflection of the other about the center of $P$. Thus $C(P)$ is well-defined, modulo $G(P)$.

If $P$ has type $\{p, q\}$ then $C(P)$ will have type $\{p_{C}, q\}$, where $p_{C} = 2p, p,$ or $p/2$ respectively, according to whether $p$ is odd; or $p$ is even and no face boundary of $P$ contains two opposite vertices; or $p$ is even and a face boundary of $P$ contains opposite vertices. For further details, see \cite{cut}.

Since the sum of lengths of corresponding edges of $P$ and $C(P)$ is the distance between opposite vertices, we have that if $P$ has directed type then $C(P)$ has bicolor type, and vice versa.

Clearly $C(C(P)) = P$ (up to a point reflection of $P$) so that $C(P)$ is a duality, which in general differs from Petrie duality and always differs from vertex/face duality. All regular polyhedra of index $2$ have a corresponding $C(P)$-dual, as well as a Petrie-dual, and so the regular polyhedra of index $2$ can be put in nine symmetry groupings, as is shown in Figures \ref{fig2}--\ref{fig4}, together with Figures 2--7 of \cite{cs}.\\

\begin{acknowledgement}
I am very grateful to Egon Schulte for his detailed and insightful appraisals during the writing of \cite{cut}. Many of his comments have been incorporated into this paper.
\end{acknowledgement}

\end{document}